\documentclass[11pt]{amsart}

\usepackage{a4wide, amsmath, amsfonts, amssymb, amsthm, breqn}
\usepackage[hyperfootnotes=false]{hyperref}
\usepackage{color}
\usepackage{float}
\usepackage{enumerate}

\usepackage{tikz}
\usetikzlibrary{decorations}
\usetikzlibrary{decorations.pathreplacing}

\usepackage{algorithm}
\usepackage[noend]{algpseudocode}
\algrenewcommand{\algorithmiccomment}[1]{\hskip3em// \textit{#1}}
\DeclareMathSymbol{\mathdblquotechar}{\mathalpha}{letters}{`"}
\usepackage{multirow}
\usepackage{makecell}
\usepackage{bm}
\usepackage{arydshln}

\usepackage{shuffle}

\usepackage{multirow}
\usepackage{makecell}
\usepackage{bm}
\usepackage{arydshln}
\usepackage{booktabs}

\allowdisplaybreaks[2]

\numberwithin{equation}{section}

\newtheorem{theorem}{Theorem}[section]
\newtheorem{lemma}[theorem]{Lemma}
\newtheorem{corollary}[theorem]{Corollary}
\newtheorem{remark}[theorem]{Remark}
\newtheorem{proposition}[theorem]{Proposition}
\newtheorem{definition}[theorem]{Definition}

\renewcommand{\epsilon}{\varepsilon}

\title[OSP and Higher Rank Signatures]{Optimal Stopping via Distribution Regression: a Higher Rank Signature Approach}

\author[Horvath]{Blanka Horvath}
\address{Blanka Horvath, Technical University of Munich, The Munich Data Science Institute, \& The Alan Turing Institute}
\email{blanka.horvath@tum.de}

\author[Lemercier]{Maud Lemercier}
\address{Maud Lemercier, University of Oxford \& The Alan Turing Institute}
\email{maud.lemercier@warwick.ac.uk}

\author[Liu]{Chong Liu}
\address{Chong Liu, ShanghaiTech University}
\email{liuchong@shanghaitech.edu.cn}

\author[Lyons]{Terry Lyons}
\address{Terry Lyons, University of Oxford \& The Alan Turing Institute}
\email{tlyons@maths.ox.ac.uk}

\author[Salvi]{Cristopher Salvi}
\address{Cristopher Salvi, Imperial College London \& The Alan Turing Institute}
\email{c.salvi@imperial.ac.uk}

\date{\today}

\begin{document}

\begin{abstract}
  Distribution Regression on path-space refers to the task of learning functions mapping the law of a stochastic process to a scalar target. The learning procedure based on the notion of path-signature, i.e. a classical transform from rough path theory, was widely used to approximate weakly continuous functionals, such as the pricing functionals of path--dependent options' payoffs. However, this approach fails for Optimal Stopping Problems arising from mathematical finance, such as the pricing of American options, because the corresponding value functions are in general discontinuous with respect to the weak topology. In this paper we develop a rigorous mathematical framework to resolve this issue by recasting an Optimal Stopping Problem as a higher order kernel mean embedding regression based on the notions of higher rank signatures of measure--valued paths and adapted topologies. The core computational component of our algorithm consists in solving a family of two--dimensional hyperbolic PDEs.
\end{abstract}

\maketitle




\section{Introduction}

The notion of signature of a path was first introduced in 1954 by K. T. Chen \cite{Chen1954iterated} and has been deployed in a variety of mathematical contexts, in particular in the study of dynamical systems driven by irregular signals. 

Given a continuous path $x:[0,T] \to \mathbb{R}^d$ with bounded variation, the signature $S(x)$ consists of the collection of iterated integrals of $x$ forming a sequence of tensors in the space $\prod_{k=0}^\infty (\mathbb{R}^d)^{\otimes k}$ defined as follows
\begin{equation}\label{eqn:signature}
	S(x) = \left(1, \int_{0<u_1<T} dx_{u_1}, ..., \int_{0<u_1<...<u_k<T} dx_{u_1}\otimes...\otimes dx_{u_k}, ... \right),
\end{equation}
where $\otimes$ denotes the standard tensor product.

Hambly and Lyons \cite{Hambly2010uniqueness} showed in 2010 that the signature characterises a path up to an equivalence relation on path-space (called tree-like equivalence), making it an injective map on the the corresponding set of equivalence classes commonly referred to as unparameterised paths. More recently, the signature has become a popular tool in the context of data science as its rich algebraic structure allows to build a powerful approximation theory for learning in presence of time series data \cite{arribas2020sig,kidger2019deep,lemercier2021siggpde}. In effect, a straightforward application of the Stone-Weierstrass theorem yields that any continuous, real-valued functions $f: K \to \mathbb{R}$ on a compact set $K$ of unparameterised (resp. parametrised) paths can be approximated uniformly well by a linear functional on $\bigoplus_{k=0}^\infty (\mathbb{R}^d)^{\otimes k}$ acting on the signature (resp. time--augmented signature) \cite{lyons2014rough}, i.e. for any $x \in K$
\begin{equation}\label{eq: universality of signature}
	f(x) \approx \langle \ell, S(x) \rangle.
\end{equation}
In the context of mathematical finance, this universal approximation property of the signature allows us to recast many option pricing problems with path-dependent payoffs as a linear regression on the expected signature of the price process. To give an idea of how this approach might work and without specifying any technical assumption for now, suppose that $f$ is the payoff function of some path-dependent option and $X$ is some price process, then the corresponding option price $\mathbb{E}[f(X)]$ can be approximated by performing linear regression on the expected signature \cite{Arribas2018pricing}:
\begin{equation}\label{eq:linear regression of expected signature}
	\mathbb{E}[f(X)] \approx \langle \ell, \mathbb{E}[S(X)] \rangle.
\end{equation}
Equation (\ref{eq:linear regression of expected signature}) can be interpreted as a regression on the law of the stochastic process $X$, and is an instance of Distribution Regression on path-space, that has been studied extensively in \cite{lemercier2021distribution} for approximating weakly continuous, real-valued functions on path-valued random variables. The core tool deployed in the above work is the signature kernel \cite{cass2020computing} defined as the inner product of two signatures and that will be extensively discussed in the sequel of this paper.

Although the above approach has been successfully applied for the pricing of financial derivatives, it has a major limitation which prevents its application to tackle an important class of optimization problems from mathametical finance, that is Optimal Stopping Problems, such as the pricing of American options.  Indeed, in such problems, the value function $v(\cdot)$ 
$$
v(\mathbb{P}) = \sup\{\mathbb{E}_{\mathbb{P}}[\gamma(X_\tau)]: \tau \text{ is stopping time} \}
$$
as a function on $\mathcal{P}(\mathcal{X})$ (the set of probability measures on pathspace $\mathcal{X}$), can neither be expressed as the supremum of expectations of continuous path--dependent functions (because stopping times are in general not continuous functions on pathspace) nor be formulated as a continuous function for the \textit{weak topology}, see a counterexample in Figure 1. This is because the weak topology completely ignores the filtration, which plays a crucial role in Optimal Stopping Problems. Hence, the Distribution Regression approach from \cite{lemercier2021distribution} cannot be used for pricing American options due to the lack of continuity of the value function $v(\cdot)$ with respect to the weak topology.\\ 
\begin{figure}
	\begin{center}
		\begin{tikzpicture}[shorten >=1pt,draw=black!50,scale=0.62]
			\draw[black, fill=black, thick] (0,0) circle (2pt);
			\draw[black, fill=black, thick] (3,.2) circle (2pt);
			\draw[black, fill=black, thick] (3,-.2) circle (2pt);
			\draw[black, thick] (0,0) -- (3,.2);
			\node () at (1.5,.4) {\footnotesize$ p=0.5 $};
			\draw[black, thick] (0,0) -- (3,-.2);
			\node () at (1.5,-.5) {\footnotesize$ p=0.5 $};
			\draw [black, thick, decorate,decoration={brace,amplitude=4pt}] (3.15,0.2) -- (3.15,-.2) node [black,midway,xshift=9pt] {\footnotesize$\frac{2}{n}$};
			\draw[black, thick] (3,.2) -- (6,1);
			\node () at (4.5,.9) {\footnotesize$ p=1 $};
			\draw[black, thick] (3,-.2) -- (6,-1);
			\node () at (4.5,-1) {\footnotesize$ p=1 $};
			\draw[black, fill=black, thick] (6,1) circle (2pt);
			\draw[black, fill=black, thick] (6,-1) circle (2pt);
			
			\draw[black, fill=black, thick] (7,0) circle (2pt);
			\draw[black, fill=black, thick] (10,0) circle (2pt);
			\draw[black, thick] (7,0) -- (10,0);
			\node () at (8.5,.25) {\footnotesize$ p=1 $};
			\draw[black, thick] (10,0) -- (13,1);
			\node () at (11.5,.90) {\footnotesize$ p=0.5 $};
			\draw[black, thick] (10,0) -- (13,-1);
			\node () at (11.5,-1.) {\footnotesize$ p=0.5 $};
			\draw[black, fill=black, thick] (13,1) circle (2pt);
			\draw[black, fill=black, thick] (13,-1) circle (2pt);
		\end{tikzpicture}
	\end{center}
	\caption{\small $\mathbb{P}_n$ (left) converges to $\mathbb{P}$ (right) weakly, but $v(\mathbb{P}_n) \nrightarrow v(\mathbb{P})$.}
	\label{fig:example}
\end{figure}
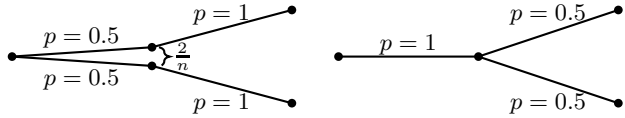
In this present paper we resolve this problem by making use of the adapted topology which was initiated by D. Aldous \cite{Aldous81adapt}, later developed by Hoover and Keisler \cite{Hoover84adapt}, and recently has been deeply studied by Beiglb\"ock et al \cite{backhoff2020all}, \cite{Backhoff2020distance},  \cite{Bartl2021Wasserstein}, \cite{Gudi2022adapted}, together with the ``higher rank signature'' approach proposed in \cite{bonnier2020adapted}. Roughly speaking, our new scheme consists of the following steps.
\begin{enumerate}
	\item \textbf{New input space and new topology:}  We adopt the suggestion from \cite{Bartl2021Wasserstein} and consider the input space (where the value function of an Optimal Stopping problem is defined) as the set of all filtered processes $\mathbb{X} = (X, \mathbb{F}, \mathbb{P})$ on discrete--time interval equipped with the so--called adapted topology $\hat \tau_1$ (i.e., the extended weak convergence introduced by \cite{Aldous81adapt}, see also \cite{bonnier2020adapted}). This topology $\hat \tau_1$ is induced by the weak convergence of the \textit{measure--valued martingale} $\hat X^1_t = \mathbb{P}[X \in \cdot|\mathcal{F}_t]$ and is stronger than the weak topology so that it can also reflect the difference of filtrations attached to underlying processes; in particular, the value function $v(\cdot)$ of the Optimal Stopping Problem now is continuous for the topology $\hat \tau_1$.
	\item \textbf{Higher rank signatures as universal features:} On the market $\mathbb{X}$ we apply the rank $2$ signature $S^2$ introduced in \cite{bonnier2020adapted}, which has the same universal properties as the usual signature of $\mathbb{R}^d$--valued paths, to the \textit{measure--valued martingale} $\hat X^1_t = \mathbb{P}[X \in \cdot|\mathcal{F}_t]$.
	\item \textbf{Distribution Regression with higher rank signatures:} we show that now the price of American option can be approximated by performing Kernel Regression on the image of (expected) rank $2$ signature $S^2$.
	
\end{enumerate}
Additionally, another contribution of the present paper is providing a feasible and explicit algorithm to implement the above theoretical framework numerically. In particular, by using the so--called signature kernel PDE trick \cite{cass2020computing}, which is based on the dynamics of signatures as collections of iterated integrals, together with the higher order kernel mean embedding method studied in \cite{Salvi2021HighRank}, we can construct a consistent estimator $\hat v(\cdot)$ for $v(\cdot)$ without using truncation of signatures, which eventually reduces the above infinite dimensional regression problem to solving a family of simple two--dimensional hyperbolic linear PDEs. \\
We summarize the main results obtained in the present paper into the following proposition, which will be proved at the end of this paper, see Appendix \ref{appendix: Proof}.
\begin{proposition}\label{prop: main result}
	Let $\mathbb{X} = (X, \mathbb{F}, \mathbb{P})$ be a discrete--time stochastic process taking values in a compact subset of $\mathbb{R}^d$ and adapted to the natural filtration $\mathbb{F}$\footnote{In fact our result also holds for general filtration, see Appendix \ref{sect: higher rank adapted topologies}.} with a underlying probabilistic model $\mathbb{P}$.  Let $v(\mathbb{X}) = \sup\{\mathbb{E}_{\mathbb{P}}[\gamma(X_\tau)]: \tau \text{ is } \mathbb{F}\text{--stopping time} \}$ be the value function of Optimal Stopping Problem attached to a bounded continuous function $\gamma$. Then we have:
	\begin{enumerate}
		\item There exists a family of consistent empirical estimator $\hat v_{(x^i)_{i\le m}}$, $m \ge 1$ for $v(\cdot)$, which are constructed by Distribution Regression with higher rank signature\footnote{this notion will be precisely defined in Section \ref{sect: preliminaries}.} such that
		$$
		\lim_{m \to \infty} \hat v_{(x^i)_{i\le m}}(\mathbb{X}) = v(\mathbb{X})
		$$
		in probability for random generations of i.i.d. samples $(x^i)_{i \le m} \sim \mathbb{P}_{X}$.
		\item The computation of $\hat v_{(x^i)_{i\le m}}$ boils down to solving a family of two--dimensional linear PDEs and the coefficient of each equation can be computed by evaluating the Euclidean inner product of $\mathbb{R}^d$--vectors.
	\end{enumerate}
\end{proposition}
To our best knowledge, this is the first work on applying kernel regression to solve Optimal Stopping Problem in terms of signature kernels and adapted topology. The only existing relevant work in this direction is the recent work \cite{Bayer2021stopping}. In that paper the authors reformulate the Optimal Stopping problem as a deterministic optimization problem on the linear functionals of the usual expected signatures, which is based on the method of ``signature stopping policies'' introduced by \cite{Kalsi2020Optimal} and is significantly different from our approach as their algorithm was not designed from the perspective of dealing with continuous functions for adapted topologies. Furthermore, our algorithm still inherits all advantages of classical signature kernel learning approaches: for instance, it works well for high dimensional stochastic processes.\\
Finally, note that although we choose Optimal Stopping Problem as the example for the application of this algorithm, the methodology developed in this paper can also be applied to approximate other types of functions on stochastic processes that are bounded and continuous for adapted topologies: for example the ones arising in the context of utility maximization problems 
studied in \cite{backhoff2020finance} or the target function in the utility maximization problem with proportional transaction costs considered in \cite{bayraktar2020finance}, provided the underlying utility functions are bounded.

\medskip

\noindent\textbf{Outline:} In Section \ref{sect: preliminaries} we introduce adapted topologies and the definition of higher rank signatures along with their universality with respect to adapted topologies, and then illustrate how they can be used to form a Distribution Regression scheme for solving Optimal Stopping problem. Section \ref{sect: Regression algorithm} is concerned with an explicit algorithm for implementing our theoretical framework into practice, a consistent estimator for approximating the target function will be constructed. In Section \ref{sect: examples} we provide two concrete examples for pricing American options under the multi--dimensional Black--Scholes model and stopped fractional Brownian motion model to justify the validity of the algorithm. Some basic knowledge on signatures and kernel learning and all proofs for claims in Sections \ref{sect: preliminaries} and \ref{sect: Regression algorithm} are recorded in Appendix, and one can also find a short discussion on applying the algorithm from Section \ref{sect: Regression algorithm} to solve Optimal Stopping problem with general filtrations in Appendix \ref{sect: higher rank adapted topologies}. Brief concluding remarks and outlooks for the future work are given in Section \ref{sect: summary and future work}.\\

\medskip

\noindent\textbf{Acknowledgment:} BH and CL acknowledge financial support from the 2021 Seed Fund project \href{SyBenDaFin}{https://www.mdsi.tum.de/en/mdsi/research/seed-funds-2021/sybendafin/} of the Munich Data Science Institute. ML, TL and CS were supported by DataSig under the EPSRC grant EP/S026347/1.

\section{Theoretical Background}\label{sect: preliminaries}
For the reader's convenience, we first introduce the notations which will be used throughout the whole paper.
\begin{itemize}
	\item  $I= \{0= t_0 < t_1 < \ldots < t_N = T\}$ is a given discrete time interval.
	\item $(\Omega,\mathcal{F},\mathbb{F}=(\mathcal{F}_t)_{t \in I},\mathbb{P})$ denotes a filtered probability space, and $X = (X_t)_{t \in I}$ denotes an $\mathbb{R}^d$--valued $\mathbb{F}$--adapted process defined on $\Omega$, i.e., $X_t$ is $\mathcal{F}_t$--measurable for every $t$.  Following the notation from \cite{Bartl2021Wasserstein} we call the five tuple
	$$
	\mathbb{X} := (\Omega,\mathcal{F},\mathbb{F}=(\mathcal{F}_t)_{t \in I},\mathbb{P}, X)
	$$
	a \textit{filtered} (or \textit{adapted}) process.
	\item $\mathbb{P}_X$ denotes the law of $X$, i.e., $\mathbb{P}_X = \mathbb{P} \circ X^{-1}$.
	\item $\mathcal{F}\mathcal{P}_I$ denotes the collection of all filtered processes indexed by $I$.
	\item $\Lambda_{\text{plain}} \subset \mathcal{F}\mathcal{P}_I$ denotes the subset of $\mathcal{F}\mathcal{P}_I$ consisting of filtered processes with \textit{natural filtration}, that is, if $\mathbb{X} := (\Omega,\mathcal{F},\mathbb{F}=(\mathcal{F}_t)_{t \in I},\mathbb{P}, X) \in \Lambda_{\text{plain}}$, then the filtration $\mathbb{F}$ is generated by the process $X$.
	\item $\mathcal{K} \subset \mathbb{R}^d$ denotes a compact subset of $\mathbb{R}^d$, and  $\mathcal{F}\mathcal{P}_I(\mathcal{K})$ denotes the set of all filtered processes taking values in $\mathcal{K}$. The notation $\Lambda_{\text{plain}}(\mathcal{K})$ is then self--explanatory.
	\item Given an $\mathbb{X} := (\Omega,\mathcal{F},\mathbb{F}=(\mathcal{F}_t)_{t \in I},\mathbb{P}, X)$, we use $\mathcal{T}_{\mathbb{X}}$ to denote the set of all stopping times (bounded by $T$) with respect to the filtration $\mathbb{F}$; that is, $\tau \in \mathcal{T}_{\mathbb{X}}$ iff $\{\tau \le t\} \in \mathcal{F}_t$ holds for every $t \in I$.
	\item For any Polish space $\mathcal{X}$, $\mathcal{P}(\mathcal{X})$ denotes the set of (Borel) probability measures on $\mathcal{X}$, which will always be equipped with the weak topology. 
	\item $C_b(\mathcal{X})$ denotes the space of real--valued bounded continuous functions defined on $\mathcal{X}$, which is always endowed with the strict topology generated by seminorms $p_{\psi}(f) := \sup_{x \in \mathcal{X}}|f(x)\psi(x)|$ for $\psi$ functions vanishing at infinity (see \cite[Definition 8]{chevyrev2018signature}). If $\mathcal{X}$ is compact, then the strict topology coincides with the uniform convergence topology.
\end{itemize}

\subsection{Prediction Processes and Adapted Topologies}
As we have seen in the introduction, the weak topology is too coarse for applications in stochastic optimal control problem such as Optimal Stopping Problem (OSP). Hence, as the first step, we need to find a stronger topology on $\mathcal{F}\mathcal{P}_I$ which gives continuity of the value function $v(\cdot)$ of OSP. Towards this aim, we introduce
the following notions which were defined by D. Aldous in \cite{Aldous81adapt} and will be the central objects in the present paper.
\begin{definition}\label{def: Prediction process}
	For any $\mathbb{X} \in \mathcal{F}\mathcal{P}_I$, where $I = \{0= t_0 < t_1 < \ldots < t_N = T\}$, the $\mathcal{P}((\mathbb{R}^d)^I)$--valued discrete process $\hat X = (\hat X_t)_{t \in I}$ such that $\hat X_t = \mathbb{P}[X \in \cdot|\mathcal{F}_t]$ is called the prediction process of $\mathbb{X}$.
\end{definition}

\begin{remark}\label{remark: basic properties of prediction process}
	
	The notion of prediction process can be extended to stochastic processes with c\`adl\`ag trajectories on continuous time interval. For a detailed introduction of prediction processes for the continuous time case, in particular their existence and uniqueness, we refer readers to \cite[Sect. 13]{Aldous81adapt}. 
	
\end{remark}

Clearly, from the above definition,  $\hat X_t = \mathbb{P}[X \in \cdot| \mathcal{F}_t]$ describes how well one can predict the trajectories of $X$ given the information $\mathcal{F}_t$ and therefore the prediction process $\hat X = (\hat X_t)_{t \in I}$ reveals the evolution style of the filtration $\mathbb{F}$ in terms of how its ``prediction ability'' varies in time. On the other hand, if $\mathcal{F}_0$ is trivial, then $\hat X_0 = \mathbb{P}_X$ is the law of $X$; that is, the law of $X$ corresponds only to the initial value of the prediction process $\hat X$. Based on the above observation, we may believe that one can obtain a deeper insight of the essential difference of the information flows induced by two filtered processes via investigating their prediction processes, whilst these differences cannot be reflected by comparing their laws at all. Therefore we introduce the following strengthened version of weak topology:

\begin{definition}\label{def: extended weak convergence}
	\begin{enumerate}
		\item We say that two filtered processes $\mathbb{X}, \mathbb{Y} \in \mathcal{F}\mathcal{P}_I$ are synonymous if their prediction processes $\hat X$ and $\hat Y$ have the same laws:
		$$ 
		\mathcal{L}(\hat X) = \mathcal{L}(\hat Y).
		$$
		\item We say that a sequence of filtered processes $\mathbb{X}^n$ converges to $\mathbb{X}$ in the sense of  extended weak convergence if 
		$$
		\mathcal{L}(\hat X^n) \to \mathcal{L}(\hat X),
		$$
		where the convergence happens in the weak topology of $\mathcal{P}(\mathcal{P}((\mathbb{R}^d)^I)^I)$.
	\end{enumerate}
\end{definition}
The next theorem confirms that extended weak convergence topology is really stronger than the usual weak topology in the sense that it gives a continuity of value functions of OSP when we are using natural filtration. Recall that $\Lambda_{\text{plain}} \subset \mathcal{F}\mathcal{P}_I$ denotes the subset consisting of all filtered processes equipped with their natural filtration.
\begin{theorem}\label{theorem: continuity of OST for extended weak convergence}
	Let $\gamma: \mathbb{R}^d \to \mathbb{R}$ be a continuous bounded function, and $v: \mathcal{F}\mathcal{P}_I \to \mathbb{R}$ be the value function of the following Optimal Stopping Problem:
	$$
	v(\mathbb{X}) := \sup_{\tau \in \mathcal{T}_{\mathbb{X}}}\mathbb{E}_{\mathbb{P}}[\gamma(X_{\tau})].
	$$
	Then
	$v$ is continuous for the extended weak convergence topology on $\Lambda_{\text{plain}}$, that is, if $\mathbb{X}^n$ and $\mathbb{X}$ are elements in $\Lambda_{\text{plain}}$ such that $\mathcal{L}(\hat X^n) \to \mathcal{L}(\hat X)$, then $v(\mathbb{X}^n) \to v(\mathbb{X})$.
	\end{theorem}
\begin{proof}
	The statement (1) was proved in \cite[Theorem 1.2]{backhoff2020all}, and the second statement follows from \cite[Theorem 17.2]{Aldous81adapt}.
\end{proof}
Here we note that the assumption that filtrations are the natural ones in the above theorem cannot be removed. In fact, even for the discrete case $I  = \{0=t_0 <t_1 < \ldots < t_N=T\}$ with $N \ge 2$, one can construct a continuous bounded function $\gamma$ and two filtered processes $\mathbb{X}, \mathbb{Y}$ with non--natural filtrations such that $\mathbb{X}$ and $\mathbb{Y}$ are synonymous (i.e., $\mathcal{L}(\hat X) = \mathcal{L}(\hat Y)$), but $v(\mathbb{X}) \neq v(\mathbb{Y})$, see \cite[Sect.6]{Bartl2021Wasserstein}. This counterexample illustrates that in general the weak topology induced by prediction processes is still not enough to describe full details of filtrations attached to filtered processes, and we still have motivations to improve the extended weak convergence topology. For this purpose we introduce the following definitions, which was initiated by Hoover and Keisler (\cite{Hoover84adapt}) based on the philosophy that more knowledge regarding filtrations can be revealed by an iteration of taking conditional expectations, see also \cite{bonnier2020adapted} and \cite{Bartl2021Wasserstein}: 
\begin{definition}\label{def: higher rank prediction processes} 
	\begin{enumerate}
		\item For $r \ge 1$ and $\mathbb{X} \in \mathcal{F}\mathcal{P}_I$, the rank $r$ prediction process $\hat X^r$ is defined recursively as the prediction process of the rank $r-1$ prediction process, i.e.,
		$$
		\hat X^r_t = \mathbb{P}[\hat X^{r-1} \in \cdot| \mathcal{F}_t], \quad  t \in I.
		$$
		By convention we set $\hat X^0 = X$.
		\item For $r \ge 1$ and $\mathbb{X}^n, \mathbb{X} \in \mathcal{F}\mathcal{P}_I$, we say that $\mathbb{X}^n$ converges to $\mathbb{X}$ in the rank $r$ adapted topology if the weak convergence holds for the law of their rank $r$ prediction processes, i.e., 
		$$
		\mathcal{L}(\hat X^{n,r}) \to \mathcal{L}(\hat X^r).
		$$
		We use $\hat \tau_r$ to denote the rank $r$ adapted topology on $\mathcal{F}\mathcal{P}_I$, and $\mathbb{X}^n \xrightarrow{r} \mathbb{X}$ to denote the corresponding convergence. Moreover, we use $\mathbb{X} \sim_r \mathbb{Y}$ to denote the relation that $\mathcal{L}(\hat X^r) = \mathcal{L}(\hat Y^r)$, and call that $\mathbb{X}$ and $\mathbb{Y}$ have the same adapted distribution of rank $r$.
	\end{enumerate}
\end{definition}
Clearly, for $r = 0$ the rank $0$ adapted topology $\hat\tau_0$ is nothing but the usual weak convergence (or convergence in law), while for $r=1$ that the rank $1$ adapted topology $\hat\tau_1$ gives Aldous' extended weak convergence.

\medskip

For technical reasons, from now on, unless explicitly stated, we will mainly consider the discrete time filtered processes taking values in a compact subset $\mathcal{K} \subset \mathbb{R}^d$ and equipped with their natural filtration, namely elements in $\Lambda_{\text{plain}}(\mathcal{K}) \subset \mathcal{F}\mathcal{P}_I(\mathcal{K})$ for $I =\{0=t_0 <t_1 < \ldots < t_N=T\}$. We believe this is the most relevant case for designing any numerical scheme for solving Optimal Stopping Problem (OSP): First, in most financial applications the filtration attached to stochastic processes are generated by themselves, i.e., the filtrations are natural; secondly, if we want to compute the value function of OSP for continuous time models numerically, we usually need to discretize them in time and localize them in space, and the resulting discretized models are indeed elements in $\Lambda_{\text{plain}}(\mathcal{K}) \subset \mathcal{F}\mathcal{P}_I(\mathcal{K})$, and in Appendix \ref{subsect: discretize} we can see that the value functions of OSP on these discrete models converge to their counterparts for continuous time markets, at least from a theoretical perspective.   \\
As a consequence, by Theorem \ref{theorem: continuity of OST for extended weak convergence}, in this case we only need to consider the extended weak convergence, i.e., the rank $1$ adapted topology $\hat \tau_1$.
\begin{remark}\label{remark: counterexample of OSP for general filtrations}
	Although in the present paper we will mainly focus on filtered processes equipped with natural filtration, let us mention here a remarkable result obtained in \cite{Bartl2021Wasserstein}: for the discrete case $I = \{0=t_0 <t_1 < \ldots < t_N=T\}$, in the contrast to Theorem \ref{theorem: continuity of OST for extended weak convergence} on adapted processes with natural filtration, the correct topology for \textbf{any} value function $v$ of OSP on the whole set $\mathcal{F}\mathcal{P}_I$ (i.e., the set of all adapted processes with general filtration) to be continuous is the rank $N$ adapted topology $\hat \tau_N$; that is, if $\mathbb{X}^n \xrightarrow{N} \mathbb{X}$ then $v(\mathbb{X}^n) \to v(\mathbb{X})$ and there exist a continuous bounded function $\gamma$ and a sequence of filtered processes $\mathbb{X}^n$, $n \ge 1$ such that $\mathbb{X}^n \to \mathbb{X}$ in the rank $N-1$ adapted topology for some $\mathbb{X} \in \mathcal{F}\mathcal{P}_I$ with $v(\mathbb{X}^n) \nrightarrow v(\mathbb{X})$ (where $v(\cdot)$ denotes the value function of OSP related to $\gamma$), see \cite[Theorem 6.1]{Bartl2021Wasserstein}. This means that for $N+1$ steps discrete filtered processes with general filtration, one needs to study their rank $N$ prediction processes to find out all essential information hidden inside these filtrations. A more detailed discussion on OSP for filtered processes with general filtration can be found in Appendix \ref{sect: higher rank adapted topologies}.
\end{remark}

\subsection{Rank 2 Signatures and Distribution Regression for Solving Optimal Stopping Problem  }\label{sect: signatures and rank 2 signatures}

Recall that by a classical result in probability theory, the law of bounded $\mathbb{R}$--valued random variables are uniquely determined by their moments. This claim remains valid for the multidimensional case, namely for $\mathcal{K} \subset \mathbb{R}^d$ a compact subset and $X,Y$ two $\mathcal{K}$--valued random variables whose distributions are denoted by $\mathbb{P}_X$ and $\mathbb{P}_Y$, it holds that $\mathbb{P}_X  = \mathbb{P}_Y$ iff
$$
\int_{\mathcal{K}} \exp_{\otimes}(x) \mathbb{P}_X(dx) = \int_{\mathcal{K}} \exp_{\otimes}(y) \mathbb{P}_Y(dy),
$$
where $\exp_{\otimes}$ is the tensor exponential map from $\mathbb{R}^d$ to the tensor algebra $\textbf{T}((\mathbb{R}^d))$ over $\mathbb{R}^d$, $\exp_{\otimes}(x) = \sum_{n=0}^\infty \frac{x^{\otimes n}}{n!}$ with $x^{\otimes n} = \sum_{i_1,\ldots,i_n=1}^d x^{i_1}\ldots x^{i_n}e_{i_1}\otimes \ldots\otimes e_{i_n}$ ($e_i$'s are the canonical basis of $\mathbb{R}^d$). Here we remark that $\textbf{T}((\mathbb{R}^d))$ contains a Hilbert space $\mathcal{H}^1$ (see Appendix \ref{subsect: hilbert space on tensor algebra}) such that it includes the image of $\exp_{\otimes}$ and therefore the above integrals are $\mathcal{H}^1$--valued Bochner integrals and always well--defined. Since $\int_{\mathcal{K}} \exp_{\otimes}(x) \mathbb{P}_X(dx)$ can be written as $\mathbb{E}_{\mathbb{P}}[\exp_{\otimes}(X)] = \sum_{n=0}^\infty \frac{1}{n!}\mathbb{E}_{\mathbb{P}}[X^{\otimes n}]$, we also call this integral as the moment generating function of $X$.\\
The above results on characterizing distributions in terms of moment generating functions can be further extended to the case of discrete time filtered processes taking values in compact sets. More precisely, for $I = \{0=t_0 <t_1 < \ldots < t_N=T\}$ there exists a mapping $S: (\mathbb{R}^d)^I \to T((\mathbb{R}^{d+1}))$ such that for $\mathbb{X}, \mathbb{Y} \in \Lambda_{\text{plain}}(\mathcal{K})$, one has 
\begin{equation}\label{eq: characterization of law by expected signature}
	\mathbb{X} \sim_0 \mathbb{Y} \iff \mathcal{L}(X) = \mathcal{L}(Y) \iff \mathbb{E}_{\mathbb{P}}[S(X)] = \mathbb{E}_{\mathbb{Q}}[S(Y)].
\end{equation}
Such a mapping $S$ is defined via
\begin{definition}\label{def: signature rank 0}
	For $x = (x_t)_{t \in I}$ a discrete time path defined on $I=\{0=t_0 <t_1 < \ldots < t_N=T\}$ such that $x_t \in \mathcal{K}$ for all $t$, its (time--augmented) signature $S(x)$ is defined as
	$$
	S(x) = \prod_{i=0}^N \exp_{\otimes}((t_{i}-t_{i-1}, x_{t_{i}}-x_{t_{i-1}})) \in \mathcal{H}^1, \quad t_{-1} := 0, x_{t_{-1}} :=0,
	$$
	where $\mathcal{H}^1$ is a Hilbert space defined as in Appendix \ref{subsect: hilbert space on tensor algebra}.
	For $\mathbb{X} \in \mathcal{F}\mathcal{P}_I(\mathcal{K})$ a discrete time filtered process, we call $\mathbb{E}_{\mathbb{P}}[S(X)]$ the expected signature of $\mathbb{X}$.
\end{definition}
\begin{remark}
	In fact, the notion of (expected) signature and the relation \eqref{eq: characterization of law by expected signature} can be further generalized to all continuous filtered processes defined on continuous time interval $[0,T]$ satisfying certain regularity (e.g., having bounded variation), see e.g. \cite{chevyrev2018signature} and the relevant literature therein. For a quick review of basic properties on the signature $S$ we also refer readers to Appendix \ref{appendix: signature}.
\end{remark}
In this section, we will generalize the notion of (expected) signature and the relation \eqref{eq: characterization of law by expected signature} to filtered processes with extended weak convergence topology. Recall that the extended weak convergence is given in terms of prediction process, which is a measure--valued martingale, we need to construct a ``signature'' map for measure--valued paths.\\
First let us introduce some notations: Let $$\mathcal{H}^2 := \Big\{\textbf{h} \in  \prod_{n=0}^\infty(\mathbb{R} \oplus \mathcal{H}^1)^{\hat \otimes n}: \sum_{n=0}^\infty \langle\textbf{h}_n, \textbf{h}_n\rangle_{(\mathbb{R} \oplus \mathcal{H}^1)^{\hat \otimes n}} < \infty \Big\}
$$ be the Hilbert space embedded in the product space $\prod_{n=0}^\infty(\mathbb{R} \oplus \mathcal{H}^1)^{\hat \otimes n}$, see Appendix \ref{subsect: hilbert space on tensor algebra}. We also use $\exp_1$ to denote the tensor exponential on $\mathcal{H}^1$.
\begin{definition}\label{def: rank 2 sig}
	Let $I =\{0=t_0 <t_1 < \ldots < t_N=T\}$, let $\mu = (\mu_t)_{t \in I}$ be a discrete time path such that $\mu_t \in \mathcal{P}((\mathcal{K})^I)$ is a distribution on $\mathcal{K}$ for each $t$. Let $S: (\mathbb{R}^d)^I \to \mathcal{H}^1$ be the signature map as in Definition \ref{def: signature rank 0} and set $\bar \mu_t := \int S(x) \mu_t(dx) \in \mathcal{H}^1$ to be the expected signature of $\mu_t$. Let $S_{\mathcal{H}^1}$ be the (time--augmented) signature map defined on $(\mathcal{H}^1)^I$, then the mapping $S^2: \mathcal{P}((\mathcal{K})^I)^I \to \mathcal{H}^2$,
	$$
	S^2(\mu) :=  S_{\mathcal{H}^1}(t \mapsto \bar \mu_t) = \prod_{i=0}^N \exp_{1}(t_{i} - t_{i-1}, \bar \mu_{t_{i}} - \bar \mu_{t_{i-1}}).
	$$
\end{definition}

For $\mathbb{X} \in \mathcal{F}\mathcal{P}_I(\mathcal{K})$ with prediction process is $\hat X^1$, by definition we have
$$
S^2(\hat X^1) = S_{\mathcal{H}^1}(t \mapsto \mathbb{E}[S(X) | \mathcal{F}_t]).
$$
The following theorem shows that the expectation of $S^2(\hat X^1)$ plays the role of ``moment generating function'' for adapted distribution of rank $1$ on $\mathcal{F}\mathcal{P}_I(\mathcal{K})$. One can find a generalized version of this theorem in \cite[Theorem 3]{bonnier2020adapted}. For readers' convenience we also give a short proof in this paper, see
Appendix \ref{appendix: Proof}. 

\begin{theorem}\label{thm: main theorem 1}
	For $\mathbb{X} = (X, \mathbb{F},\mathbb{P}), \mathbb{Y} = (Y, \mathbb{G}, \mathbb{Q}) \in \mathcal{F}\mathcal{P}_I(\mathcal{K})$ with $I = \{0 = t_0 < t_1 < \ldots < t_N =T\}$ we have
	\begin{itemize}
		\item $
		\mathbb{X} \sim _0\mathbb{Y} \iff \mathcal{L}(\hat X^0) = \mathcal{L}(\hat Y^0) \iff \mathbb{E}_{\mathbb{P}}[S(X)] = \mathbb{E}_{\mathbb{Q}}[S(Y)],
		$
		\item $
		\mathbb{X} \sim _1\mathbb{Y} \iff \mathcal{L}(\hat X^1) = \mathcal{L}(\hat Y^1) \iff \mathbb{E}_{\mathbb{P}}[S^{2}(\hat X^1)] = \mathbb{E}_{\mathbb{Q}}[S^{2}(\hat Y^1)],
		$
	\end{itemize}
	where $\hat X^r$ and $\hat Y^r$, $r= 0,1$ denote the rank $r$ prediction process of $\mathbb{X}$ and $\mathbb{Y}$, respectively, see Definition \ref{def: higher rank prediction processes}.
\end{theorem}
In terms of terminologies from Kernel Learning, Theorem \ref{thm: main theorem 1} tells us that the signature map $S$ and the rank $2$ signature map $S^2$ are universal feature mappings (see Appendix \ref{sect: kernel learning} for the notion of ``universal feature map'') on $\mathbb{R}^d$--valued path space $(\mathbb{R}^d)^I$ and measure--valued path space $(\mathcal{P}((\mathbb{R}^d)^I))^I$ respectively. As a consequence, they induce the following Signature Maximum Mean Discrepancies (Signature MMDs) for adapted topology of rank $0$ and rank $1$, respectively:
\begin{theorem}\label{thm: main theorem 2}
	Let $I = \{0 = t_0 < t_1 < \ldots < t_N =T\}$ and $\mathcal{K} \subset \mathbb{R}^d$ be a compact subset. Then
	\begin{enumerate}
		\item The rank $1$ Signature MMD 
		$$
		\mathcal{D}^1_{\mathcal{S}}(\mathbb{X}, \mathbb{Y}) := \|\mathbb{E}_{\mathbb{P}}[S(X)] - \mathbb{E}_{\mathbb{Q}}[S(Y)]\|_{\mathcal{H}_1}
		$$
		metrizes the rank $0$ adapted topology $\hat \tau_0$ (i.e., the usual weak topology on laws of stochastic processes) on $\mathcal{F}\mathcal{P}_I(\mathcal{K})$.
		\item The rank $2$ Signature MMD 
		$$
		\mathcal{D}^2_{\mathcal{S}}(\mathbb{X}, \mathbb{Y}) := \|\mathbb{E}_{\mathbb{P}}[S^{2}(\hat X^1)] - \mathbb{E}_{\mathbb{Q}}[S^{2}(\hat Y^1)]\|_{\mathcal{H}_2}
		$$
		metrizes the rank $1$ adapted topology $\hat \tau_1$ (i.e., the Aldous' extended weak convergence, see Definition \ref{def: extended weak convergence}) on $\mathcal{F}\mathcal{P}_I(\mathcal{K})$.
	\end{enumerate}
\end{theorem}
One can find a proof of the above theorem in \cite[Theorem 4]{bonnier2020adapted}. For reader's convenience, We will give a short proof of Theorem \ref{thm: main theorem 2} in Appendix \ref{appendix: Proof}.

\begin{remark}\label{remark: a discussion on general filtration cases}
	In fact, given the above rank $2$ signature map $S^2$, one can construct the rank $3$ signature map $S^3$ for rank $2$ prediction processes in terms of $S^2$ and the (time--augmented) signature mapping $S_{\mathcal{H}^2}$ on paths valued in $\mathcal{H}^2$ exactly as we construct $S^2$ out of $S = S^1$. Continuing the procedure, we can recursively obtain rank $r$ signature maps for rank $r-1$ prediction processes and the resulting rank $r$ Signature MMD $\mathcal{D}_{\mathcal{S}}^r$ metrizes the rank $r-1$ adapted topology for all $r \ge 1$, see Appendix \ref{sect: higher rank adapted topologies} for more details.
\end{remark}

\begin{remark}\label{remark: relation with adapted Wasserstein metric}
	Another popular metric for characterizing adapted topologies is the so--called adapted (or, bi--causal) Wasserstein distance (see e.g. \cite{backhoff2020all} and \cite{Bartl2021Wasserstein}) which reads that
	$$
	\mathcal{A}\mathcal{W}_p(\mathbb{X}, \mathbb{Y}) = \inf_{\pi \in \text{CPL}_{\text{bc}}(\mathbb{X}, \mathbb{Y})} \int d_{(\mathbb{R}^d)^I} (x, y)^p \pi(dx, dy),
	$$
	where $\text{CPL}_{\text{bc}}$ is the set of bi--causal couplings such that the constraints regarding filtrations are contained in this bi--causality requirement, see \cite[Definition 2.1]{Bartl2021Wasserstein}. In the seminal work \cite{backhoff2020all} the authors proved that for $I = \{0 = t_0 < t_1 < \ldots < t_N = T\}$ the extended weak convergence topology $\hat \tau_1$ is equivalent to the topology induced by $\mathcal{A}\mathcal{W}_p$ on $\Lambda_{\text{plain}}(\mathcal{K})$, which means that the rank $2$ Signature MMD $\mathcal{D}^2_{\mathcal{S}}$ generates the same topology as adapted Wasserstein metric on $\Lambda_{\text{plain}}(\mathcal{K})$. However, this equivalence fails when the filtrations are not natural, see \cite{Bartl2021Wasserstein}. In this case the adapted Wasserstein metric is equivalent to rank $N+1$ Signature MMD in the sense that they generate the same topology $\hat \tau_N$, see Theorem \ref{thm: equivalence of adapted wasserstein metric and higher rank MMD} in Appendix \ref{sect: higher rank adapted topologies}. One can find a nice survey of various topologies which can characterize the difference of filtrations (including adapted topologies and the topology metrized by the adapted Wasserstein distance) in the recent preprint \cite{Gudi2022adapted}.
\end{remark}
Since the value function in Optimal Stopping Problem (OSP)
$$
v(\mathbb{X}) := \sup_{\tau \in \mathcal{T}_{\mathbb{X}}}\mathbb{E}_{\mathbb{P}}[\gamma(X_{\tau})].
$$
is continuous on $\Lambda_{\text{plain}}$ with respect to the Aldous' extended weak convergence (see Theorem \ref{theorem: continuity of OST for extended weak convergence}), using Theorem \ref{thm: main theorem 2} above we obtain that $v(\mathbb{X}^n) \to v(\mathbb{X})$ provided $\mathcal{D}^{2}_{\mathcal{S}}(\mathbb{X}^n, \mathbb{X}) \to 0$ for $\mathbb{X}^n, \mathbb{X} \in \Lambda_{\text{plain}}(\mathcal{K})$. This observation together with a classical result from Kernel Learning implies the following important claim, which is the theoretical foundation of our numerical scheme for solving Optimal Stopping Problem. We postpone the proof of the following theorem to Appendix \ref{appendix: Proof}. \\
Let us recall the notion of strict topology introduced at the beginning of Chapter 2. According to \cite{Giles1972strict} (see also \cite[Definition 8]{chevyrev2018signature}), for a topological space $\mathcal{X}$, the strict topology on $C_b(\mathcal{X},\mathbb{R})$ is the topology generated by seminorms $p_\psi(f) = \sup_{x \in \mathcal{X}}|f(x)\psi(x)|$, where $\psi$ vanishes at infinity in the sense that for any $\varepsilon>0$ there is a compact set $K \subset \mathcal{X}$ such that $\sup_{x \notin K}|\psi(x)|\le \varepsilon$. Clearly the strict topology is weaker than the uniform topology but stronger than the topology of uniform convergence on compact sets.
\begin{theorem}\label{thm: main theorem 3}
	Let $\sigma > 0$ be fixed. We define the kernel $K^2_{\mathcal{S}}: \Lambda_{\text{plain}}(\mathcal{K}) \times \Lambda_{\text{plain}}(\mathcal{K}) \to \mathbb{R}$,
	$$
	K^2_{\mathcal{S}}(\mathbb{X}, \mathbb{Y}) := \exp(-\sigma^2 \mathcal{D}^2_{\mathcal{S}}(\mathbb{X}, \mathbb{Y})^2)
	$$
	Then the Reproducing Kernel Hilbert Space (RKHS) generated by $K^2_{\mathcal{S}}$ is dense in the space $C_b(\Lambda_{\text{plain}}(\mathcal{K}),\hat \tau_{1})$ with respect to the strict topology. In particular, if $v(\cdot)$ denotes the value function in OSP, then for any $\varepsilon > 0$ and $\mathbb{X} \in \Lambda_{\text{plain}}(\mathcal{K})$, there exist an $M \in \mathbb{N}$, $\mathbb{Y}^i \in \Lambda_{\text{plain}}(\mathcal{K})$ and scalars $a_i \in \mathbb{R}$, $i=1,\ldots,M$ such that 
	$$
	\Big|v(\mathbb{X}) - \sum_{i=1}^M a_i \exp(-\sigma^2 \mathcal{D}^2_{\mathcal{S}}(\mathbb{X}, \mathbb{Y}^i)^2)\Big| \le \varepsilon,
	$$
	and this inequality holds uniformly on every compact set on $\Lambda_{\text{plain}}(\mathcal{K})$ with respect to $\hat \tau_1$.
\end{theorem}
\begin{remark}
	Note that in the above theorem the RKHS of $K^2_{\mathcal{S}}$ is only dense in the space $C_b(\Lambda_{\text{plain}}(\mathcal{K}),\hat \tau_1)$ with respect to the strict topology instead of the uniform topology. This is because the set $\Lambda_{\text{plain}}(\mathbb{R}^d)$ of stochastic processes with their natural filtrations is not complete with respect to the rank $1$ topology (see e.g. \cite{backhoff2020all} and \cite{Backhoff2020distance}) and therefore the set $\Lambda_{\text{plain}}(\mathcal{K})$ is in general not compact for the rank $1$ topology although the state space $\mathcal{K}$ is compact (this observation again reveals the difference between the usual weak topology and the higher rank adapted topology). Hence the strict topology introduced in \cite{Giles1972strict} seems to be the correct notion for the approximation to the value functions of OSP related to natural filtration. For a concrete criterion for characterizing the compactness of $\Lambda_{\text{plain}}(\mathbb{R}^d)$ we refer readers to \cite{Eder2019compactness}.
\end{remark}

\begin{remark}\label{remark: OSP for general filtered processes}
	Thanks to Remarks \ref{remark: a discussion on general filtration cases} and \ref{remark: relation with adapted Wasserstein metric} the above theorem can be generalized to rank $r$ adapted topologies for any $r \ge 2$. In particular, our kernel learning approach remains valid for OSP defined on filtered processes with general filtration. Again, we refer readers to Appendix \ref{sect: higher rank adapted topologies} or \cite{bonnier2020adapted} for a more  detailed discussion of higher rank signatures and their applications in solving OSP.
\end{remark}
The rest of this paper will be devoted to provide an efficient numerical scheme based on Theorem \ref{thm: main theorem 3} to perform a signature kernel regression of value functions $v(\cdot)$ from OSP; in particular, we will give a simple and precise way to compute the rank $2$ Signature MMD $\mathcal{D}^2_{\mathcal{S}}$. 

\section{Algorithm}\label{sect: Regression algorithm}
In view of Theorem \ref{thm: main theorem 3}, we see that the value function $v(\cdot)$ of OSP defined on $\Lambda_{\text{plain}}(\mathcal{K}) \subset \mathcal{F}\mathcal{P}_I(\mathcal{K})$ can be approximated by performing the so--called \textit{Distribution Regression} on kernels
$$
\exp(-\sigma^2 \mathcal{D}^2_{\mathcal{S}}(\cdot,\cdot)^2).
$$
In particular, to find a numerical solution for value functions $v(\cdot)$ from OSP, our task now boils down to compute the rank $2$ Signature MMD $\mathcal{D}^2_{\mathcal{S}}(\cdot,\cdot)$ numerically.\\
However, the calculation of the rank $2$ Signature MMD $\mathcal{D}^2_{\mathcal{S}}(\cdot,\cdot)$ is highly non--trivial: By the definition of $\mathcal{D}^2_{\mathcal{S}}(\cdot,\cdot)$ given in Theorem \ref{thm: main theorem 2}, namely
$$
\mathcal{D}^2_{\mathcal{S}}(\mathbb{X},\mathbb{Y}) = \|\mathbb{E}_{\mathbb{P}}[S^2(\hat X^1)] - \mathbb{E}_{\mathbb{Q}}[S^2(\hat Y^1)]\|_{\mathcal{H}^2},
$$
where $S^2(\hat X^1) = S_{\mathcal{H}^1}(t \mapsto \mathbb{E}_{\mathbb{P}}[S(X)|\mathcal{F}_t])$, one may expect to split the numerical scheme for getting $\mathcal{D}^2_{\mathcal{S}}(\mathbb{X},\mathbb{Y})$ into the following steps:
\begin{enumerate}
	\item Estimate the conditional distribution $\mathbb{P}(X \in \cdot| \mathcal{F}_t)$ for each $t \in I$ via sampling paths $X|_{I_t}$ under $\mathbb{P}$ (where $I_t := \{s \in I: s \le t\}$).
	\item Compute the conditional expectation $\mathbb{E}_{\mathbb{P}}[S(X)|\mathcal{F}_t]$ for each $t \in I$ through sampling paths according to the conditional distribution $\mathbb{P}(X \in \cdot| \mathcal{F}_t)$. 
	\item Compute the signature of the $\mathcal{H}^1$--valued path $t \mapsto \mathbb{E}_{\mathbb{P}}[S(X)|\mathcal{F}_t]$ and the expected rank $2$ signature $\mathbb{E}_{\mathbb{P}}[S^2(\hat X^1)]$ by suing sample paths of $X$ obtained in step 1.
\end{enumerate}
The most crucial weakness of the above approach are two fold: First, the sampling of the conditional distribution $\mathbb{P}(X \in \cdot| \mathcal{F}_t)$ is usually not easy; secondly, as signature maps take values in infinite dimensional spaces, one has to do truncation of signatures in order to find numerical solutions; so, if the dimension $d$ of the codomain of $X$ is high, one needs to pay a huge cost to compute truncated signatures (as the dimension of truncated signatures grows exponentially). Even worse, here we also need to consider the truncation of rank $2$ signature $S^2$ which takes values in the tensor algebra of an infinite dimensional space with very complicated algebraic structure (see \cite[Appendix B.2]{bonnier2020adapted}). As a consequence, the above approach involved with truncation procedure seems not feasible in practice. \\
In this section we will show that, thanks to the Hilbert space structures of $\mathcal{H}^1$ and $\mathcal{H}^2$ and the special dynamics satisfied by signatures $S$ and $S^2$, we can compute $\mathcal{D}^2_{\mathcal{S}}(\cdot,\cdot)$ \textit{without} sampling of conditional distributions or doing any truncation of signatures; instead, one only needs to generate samples of $X$ under $\mathbb{P}$ and solve a finite families of $2$--dimensional linear hyperbolic PDEs of the same type. The main recipe of our new algorithm is the so--called Higher Order Conditional signature kernel Mean Embedding (Higher Order CKME) introduced in \cite{Salvi2021HighRank} and the Signature PDE trick developed in \cite{cass2020computing}, as we will see below.

\subsection{Conditional Signature Kernel Mean Embedding}\label{sect: conditional KME}
In order to compute the rank $2$ Signature MMD $\mathcal{D}_{\mathcal{S}}^2(\mathbb{X},\mathbb{Y}) = \|\mathbb{E}_{\mathbb{P}}[S^2(\hat X^1)] - \mathbb{E}_{\mathbb{Q}}[S^2(\hat Y^1)]\|_{\mathcal{H}^2}$ numerically, we first reformulate this distance in terms of conditional signature kernel mean embedding (rather than using rank $2$ signatures directly), which will allow us to design an efficient numerical scheme via tools from Kernel Learning, see the next subsection.\\
Let $S$ be the (time--augmented) signature for $\mathbb{R}^d$--valued paths defined as in Definition \ref{def: signature rank 0}, the associated signature kernel $k_S$ is defined by
\begin{equation}\label{eq: first order signature kernel}
	k_S(x,y) := \langle S(x), S(y) \rangle_{\mathcal{H}^1}, \quad x,y \in (\mathbb{R}^d)^I.
\end{equation}
Let $\mathcal{H}_{\mathcal{S}}$ denote the Reproducing Kernel Hilbert Space (RKHS) with $k_S$ being the reproducing kernel. Then for filtered process $\mathbb{X} \in \Lambda_{\text{plain}}(\mathcal{K}) \subset \mathcal{F}\mathcal{P}_I(\mathcal{K})$ with $I=\{0=t_0 <t_1 < \ldots < t_N = T\}$ we define its (Signature) Kernel Mean Embedding (KME) as
$$
\mu_X := \int_{(\mathcal{K})^I} k_S(x, \cdot) \mathbb{P}_X (dx),
$$
which is a well--defined $\mathcal{H}_{\mathcal{S}}$--valued Bochner integral as $\mathcal{K}$ is compact and the signature map $S$ is continuous (see sect. \ref{subsect: analytic signature}). Then for another filtered process $\mathbb{Y} \in \Lambda_{\text{plain}}(\mathcal{K})$, using the reproducing property 
\begin{equation}\label{eq: reproducing property}
	\langle k_S(x, \cdot), k_S(y,\cdot) \rangle_{\mathcal{H}_{\mathcal{S}}} = \langle S(x), S(y) \rangle_{\mathcal{H}^1}    
\end{equation}
we obtain another formulation of $1$--st rank MMD:
\begin{align*}
	\mathcal{D}^1_{\mathcal{S}}(\mathbb{X}, \mathbb{Y})^2 &=  \|\mathbb{E}_{\mathbb{P}}[S(X)] - \mathbb{E}_{\mathbb{Q}}[S(Y)]\|_{\mathcal{H}^1}^2\\
	&= \int  \langle S(x), S(x^\prime) \rangle_{\mathcal{H}^1} \mathbb{P}_X\otimes \mathbb{P}_X(dx, dx^\prime) \\
	&\quad + \int \langle S(y), S(y^\prime) \rangle_{\mathcal{H}^1} \mathbb{Q}_Y\otimes \mathbb{Q}_Y(dy, dy^\prime)\\
	&\quad - 2\int  \langle S(x), S(y) \rangle_{\mathcal{H}^1} \mathbb{P}_X\otimes \mathbb{Q}_Y(dx, dy)\\
	&=\|\mu_X - \mu_Y\|_{\mathcal{H}_{\mathcal{S}}}^2.
\end{align*}
To obtain a similar expression for $2$--nd rank MMD, we define the Conditional Kernel Mean Embedding (CKME) of $\hat X^1_t = \mathbb{P}(X \in \cdot |\mathcal{F}_t)$ via
\begin{equation}\label{eq: CKME}
	\mu_{\hat X^1_t}   =  \int k_S(x, \cdot) \hat X^1_t(dx) = \mathbb{E}_{\mathbb{P}}[k_S(X,\cdot) |\mathcal{F}_t].
\end{equation}
Clearly $\mu_{\hat X} = (\mu_{\hat X^1_t})_{t \in I}$ is an $\mathcal{H}_{\mathcal{S}}$--valued stochastic process. Let $S_{\mathcal{H}_{\mathcal{S}}}$ be the (time--augmented) signature map on $\mathcal{H}_{\mathcal{S}}$--valued paths and $k_{S_{\mathcal{H}_{\mathcal{S}}}}$ be the associated signature kernel, which results in a ``higher rank'' RKHS $\mathcal{H}_{\mathcal{S}}^2$ with reproducing kernel $k_{S_{\mathcal{H}_{\mathcal{S}}}}$. More explicitly, $\mathcal{H}^2_{\mathcal{S}}$ is a Hilbert space contained in the space of real--valued continuous functions defined on $(\mathcal{H}_{\mathcal{S}})^I$ such that the following reproducing property holds true for all $\hat x, \hat y \in (\mathcal{H}_{\mathcal{S}})^I$:
\begin{equation}\label{eq: second order signature kernel}
	\langle k_{S_{\mathcal{H}_{\mathcal{S}}}}(\hat x, \cdot), k_{S_{\mathcal{H}_{\mathcal{S}}}}(\hat y,\cdot) \rangle_{\mathcal{H}_{\mathcal{S}}^2} = \langle S_{\mathcal{H}_{\mathcal{S}}}(\hat x), S_{\mathcal{H}_{\mathcal{S}}}( \hat y) \rangle_{\widetilde{\mathbb{R} \oplus \mathcal{H}_{\mathcal{S}}}}
\end{equation}
where the Hilbert space $\widetilde{\mathbb{R} \oplus \mathcal{H}_{\mathcal{S}}} \subset \textbf{T}((\mathbb{R} \oplus \mathcal{H}_{\mathcal{S}}))$ is defined as in Appendix \ref{subsect: hilbert space on tensor algebra}.
We define the $2$--nd  order (signature) Kernel Mean Embedding of $\mathbb{X} \in \mathcal{F}\mathcal{P}_I(\mathcal{K})$ by
$$
\mu^2_{\hat X^1} := \int_{\hat x  \in (\mathcal{H}_{\mathcal{S}})^I} k_{S_{\mathcal{H}_{\mathcal{S}}}}(\hat x, \cdot) \mathbb{P}_{\mu_{\hat X^1}}(d\hat x) \in \mathcal{H}_{\mathcal{S}}^2.
$$
\begin{lemma}\label{lemma: MMD by RKHS}
	For $\mathbb{X}, \mathbb{Y} \in \mathcal{F}\mathcal{P}_I(\mathcal{K})$, 
	$$
	\mathcal{D}^2_{\mathcal{S}}(\mathbb{X}, \mathbb{Y}) = \|\mu^2_{\hat X^1} - \mu^2_{\hat Y^1}\|_{\mathcal{H}_{\mathcal{S}}^2}.
	$$
\end{lemma}
The proof of the above lemma relies heavily on the fact that the inner products appeared in \eqref{eq: first order signature kernel} and \eqref{eq: second order signature kernel} are solutions of the so--called Goursat PDEs, see Appendix \ref{appendix: Proof}.
\begin{remark}\label{remark: generalization of CKME to higher order}
	Although we write the $2$--nd order KME as $\mu^2_{\hat X^1}$, one should keep in mind that it actually depends on the law of $\hat X^1$. Then since it holds that $\hat X^2_t$ takes image in the same space which contains $\mathcal{L}(\hat X^1)$ for each $t \in I$ by the definition of prediction processes, the $2$--nd order CKME $\mu^2_{\hat X^2_t}$ is well--defined, just like how we define $\mu_{\hat X^1}$ in \eqref{eq: CKME}. Consequently we can define the $3$--rd order KME $\mu^3_{\hat X^2}$ which takes values in $\mathcal{H}_{\mathcal{S}}^3 = ``\mathcal{H}_{\mathcal{S}} \circ \mathcal{H}_{\mathcal{S}} \circ \mathcal{H}_{\mathcal{S}}"$. Continuing this procedure we can construct $r$--th order KME $\mu^{r}_{\hat X^{r-1}}$ for any $r \ge 1$ which takes values in $\mathcal{H}_{\mathcal{S}}^{r}$, and we can prove that $\mathcal{D}^r_{\mathcal{S}}(\mathbb{X},\mathbb{Y}) = \|\mu^r_{\hat X^{r-1}} - \mu^r_{\hat Y^{r-1}}\|_{\mathcal{H}^r_{\mathcal{S}}}$ by an induction argument together with the argument used in the proof of Lemma \ref{lemma: MMD by RKHS}, where $\mathcal{D}^r_{\mathcal{S}}$ denotes the rank $r$ Signature MMD which metrizes the rank $r$ adapted topology, see Appendix \ref{sect: higher rank adapted topologies}.
\end{remark}

\subsection{Empirical Estimation of Higher order KME}\label{sect: estimation of CKME}
In this section we propose a  feasible numerical implementation of rank $2$ Signature MMD $\mathcal{D}^2_{\mathcal{S}}(\cdot,\cdot)$ in terms of $2$--nd order KME introduced in the last subsection. Here we emphasize again that all filtered processes are equipped with \textit{natural filtrations}, i.e., we will focus on elements from the set $\Lambda_{\text{plain}}(\mathcal{K})$.  This setup allows us to estimate the higher order MMD $\mathcal{D}^2_{\mathcal{S}}(\mathbb{X},\mathbb{Y})$ out of sample paths generated from $\mathbb{X}$ and $\mathbb{Y}$. \\
Recall that by Lemma \ref{lemma: MMD by RKHS} we have $\mathcal{D}^2_{\mathcal{S}}(\mathbb{X},\mathbb{Y})^2 = \|\mu^2_{\hat X^1} - \mu^2_{\hat Y^1}\|_{\mathcal{H}^2_{\mathcal{S}}}^2$ and
$$
\mu^2_{\hat X^1} = \mathbb{E}_{\mathbb{P}}[k_{S_{\mathcal{H}_{\mathcal{S}}}}(\mu_{\hat X^1}, \cdot)], \quad \mu^2_{\hat Y^1} = \mathbb{E}_{\mathbb{Q}}[k_{S_{\mathcal{H}_{\mathcal{S}}}}(\mu_{\hat Y^1}, \cdot)]
$$
with $\mu_{\hat X^1} = (\mu_{\hat X^1_t})_{t \in I}$, $\mu_{\hat X^1_t} = \mathbb{E}_{\mathbb{P}}[k_S(X,\cdot)| \mathcal{F}_t]$ and similar formulations hold for $\mu_{\hat Y^1}$.
From the above equations, we can see that $\mu^2_{\hat X^1}$ is the expectation of the kernel $k_{S_{\mathcal{H}_{\mathcal{S}}}}$ with respect to $\mathcal{L}(\mu_{\hat X^1})$. As we required that $\mathcal{F}_t = \sigma(X_s, s \le t)$, for each $t \in I$ there exists a measurable function $F_{t, k_S} : \mathcal{K}^{I_t} \to \mathcal{H}_{\mathcal{S}}$ such that 
$$
\mathbb{E}_{\mathbb{P}}[k_S(X,\cdot)| \mathcal{F}_t] = F_{t,k_S} \circ X|_{I_t}
$$
where $I_t = \{s \in I: s \le t\}$ and $X|_{I_t} = (X_s)_{s \in I_t}$ is the restriction of $X$ onto $I_t$. Hence, in order to compute $\mu^2_{\hat X^1}$, the first and the most important task is to construct a consistent estimator $\widehat{F_{t,k_S}}: \mathcal{K}^{I_t} \to \mathcal{H}_{\mathcal{S}}$ for the measurable function $F_{t,k_S}$. \\
For the construction of such an estimator we will adopt the scheme proposed in \cite{park2020measure} and \cite{park2021LS}: Roughly speaking, since $\mathbb{E}_{\mathbb{P}}[k_S(X,\cdot)| \mathcal{F}_t] = F_{t,k_S} \circ X|_{I_t}$, $F_{t,k_S}$ is the minimizer of the loss function $\mathbb{E}[\|k_S(X,\cdot) - F(X|_{I_t})\|_{\mathcal{H}_{\mathcal{S}}}^2]$ over all measurable mapping $F \in L^2(\mathcal{K}^{I_t}, \mathcal{H}_{\mathcal{S}}, \mathbb{P}_{X|_{I_t}})$, which suggests us to solve this minimization problem within this $L^2$--space. For a further reduction of the complexity, we want to consider only candidates from a smaller subspace.\\
	Towards this aim, let $k_{S_t}(\cdot,\cdot)$ be the signature kernel on $\mathcal{K}^{I_t}$, which is defined by $k_{S_t}(v, w) = \langle S(v), S(w) \rangle_{\mathcal{H}^1}$ for paths $v,w \in \mathcal{K}^{I_t}$ and can be viewed as the restriction of signature kernel $k_S(\cdot,\cdot)$ onto the pathspace defined on subinterval $I_t$.
	Then we define $\mathcal{G}^t_{\mathcal{H}_S}$ to be the $\mathcal{H}_{\mathcal{S}}$--valued RKHS induced by the  $\mathcal{H}_S$--kernel $\Gamma_t = k_{S_t}(\cdot,\cdot)\text{Id}_{\mathcal{H}_{\mathcal{S}}}$. This means that $\mathcal{G}^t_{\mathcal{H}_S}$ is a Hilbert space whose elements are functions from $\mathcal{K}^{I_t}$ to $\mathcal{H}_S$ and its inner product satisfies that
	$$
	\langle f(z), h \rangle_{\mathcal{H}_S}  = \langle f, k_{S_t}(\cdot, z)h \rangle_{\mathcal{G}^t_{\mathcal{H}_S}},  \quad \langle h, k_{S_t}(z, z^\prime)h^\prime \rangle_{\mathcal{H}_S}  = \langle k_{S_t}(\cdot, z)h, k_{S_t}(\cdot, z^\prime)h^\prime \rangle_{\mathcal{G}^t_{\mathcal{H}_S}}
	$$
	for all $h, h^\prime \in \mathcal{H}_S$ and $z, z^\prime \in \mathcal{K}^{I_t}$. 
	Thanks to the nice properties of signature kernel $k_{S_t}$, one can show that the Hilbert space $\mathcal{G}^t_{\mathcal{H}_S}$ is dense in $L^2(\mathcal{K}^{I_t}, \mathcal{H}_{\mathcal{S}}, \mathbb{P}_{X|_{I_t}})$,
	see Appendix \ref{sect: kernel learning} for more detailed explanations on these arguments regarding vector--valued RKHS.\\
Now, due to the density of $\mathcal{G}^t_{\mathcal{H}_S}$, it is enough to consider the minimization problem $\mathbb{E}[\|k_S(X,\cdot) - F(X|_{I_t})\|_{\mathcal{H}_{\mathcal{S}}}^2]$ over all elements $F \in \mathcal{G}^t_{\mathcal{H}_S}$. As a consequence, it is plausible that given i.i.d. sample paths $(\bar x^i)_{i \le M} \sim \mathbb{P}_X$, the unique minimizer of the regularized empirical surrogate loss with regularization parameter $\lambda_M >0$ (Here, we need $\lambda_M$ decays to $0$ slower than $\mathcal{O}(M^{-1/2})$ to guarantee the convergence in Theorem \ref{thm: convergence of empirical estimates} below, see \cite[Theorem 4.4]{park2021LS}) on $\mathcal{G}^t_{\mathcal{H}_{\mathcal{S}}}$
$$
\hat{\mathcal{E}}_{X|_{I_t}, M, \lambda_M}(F) := \frac{1}{M}\sum_{i=1}^M \|k_{S}(\bar x^i, \cdot) - F(\bar x^i|_{I_t})\|_{\mathcal{H}_S}^2 + \lambda_M \|F\|^2_{\mathcal{G}^t_{\mathcal{H}_{\mathcal{S}}}}
$$
should be a good estimator for the desired function $F_{t,k_S}$. 
In view of \cite[Sect. 4]{park2020measure}, the minimizer is of the form that
\begin{equation}\label{eq: minimizer for emirical loss}
	\widehat{F}_{t,k_S}((\bar x^i)_{i \le M},\cdot)  = \textbf{k}^X_t(\cdot)(\textbf{K}^X_{t,t} + M\lambda_M \text{I}_M)^{-1}\textbf{k}^X_T,
\end{equation}	
where $\textbf{k}^X_t(\cdot) = (k_{S_t}(\bar x^i|_{I_t}, \cdot))_{i=1,\ldots,M}:\mathcal{K}^{I_t} \to \mathbb{R}^M$ is an $\mathbb{R}^M$--valued (continuous) function on $\mathcal{K}^{I_t}$, $\textbf{K}^X_{t,t}= ([\textbf{K}^X_{t,t}]_{i,j})_{i,j=1,\ldots,M}$ is an $M \times M$ matrix with entries $[\textbf{K}^X_{t,t}]_{i,j} = k_{S_t}(\bar x^i|_{I_t}, \bar x^j|_{I_t})$, $\textbf{k}^X_T = (k_S(\bar x^i,\cdot))_{i=1,\ldots,M}$ is a vector in $(\mathcal{H}_{\mathcal{S}})^M$ and $\text{I}_M$ is the $M \times M$--identity matrix. Note that $\textbf{k}^X_t(\cdot)$, $\textbf{K}^X_{t,t}$ and $\textbf{k}^X_T$ are all dependent on samples paths $(\bar x^i)_{i\le M}$ (but for simplicity of notations we omit them on the right hand side in \eqref{eq: minimizer for emirical loss}), the subscript $t$ for $\textbf{k}^X_t(\cdot)$ emphasizes that this function only depends on the sample paths on sub--interval $I_t$ while the subscript $T$ for $\textbf{k}_T^X$ means that this vector depends on the whole trajectories of sample paths on $I$. In particular, we use the notation $\widehat{F_{t,k_S}}((\bar x^i)_{i \le M},\cdot)$ to emphasize the dependence of this estimator on the sample paths  $(\bar x^i)_{i \le M} \sim \mathbb{P}_X$.\\
Given i.i.d. sample paths $(\bar x^i)_{i \le M} \sim \mathbb{P}_X$, we define the $\mathcal{H}_{\mathcal{S}}$--valued stochastic process $$
\hat \mu_{\hat X^1_t}((\bar x^i)_{i \le M},\cdot) := \widehat{F_{t,k_S}}((\bar x^i)_{i \le M},\cdot), 
$$
$t \in I$ (defined on the pathspace $\mathcal{K}^I$). The following theorem justifies that $\hat \mu_{\hat X^1_t}((\bar x^i)_{i \le M},\cdot)$ is indeed a consistent estimator of our CKME process $\mu_{\hat X^1_t} = \mathbb{E}_{\mathbb{P}}[k_S(X,\cdot)|\mathcal{F}_t]$, $t \in I$. 
\begin{theorem}\label{thm: convergence of empirical estimates}
	For any $\mathbb{X} \in \Lambda_{\text{plain}}(\mathcal{K})$, it holds that
	$$
	\lim_{M \to \infty} \|\mathbb{E}_{\mathbb{P}_X}[k_{S_{\mathcal{H}_{\mathcal{S}}}}(t \mapsto \hat \mu_{\hat X^1_t}((\bar x^i)_{i \le M},\cdot))] - \mu^2_{\hat X^1}\|_{\mathcal{H}^2_{\mathcal{S}}} = 0
	$$
	in probability for i.i.d. realizations of $(\bar x^i)_{i \le M} \sim \mathbb{P}_X$, provided $\lambda_M$ decays to $0$ at a slower rate than $\mathcal{O}(M^{-1/2})$.
\end{theorem}	
\begin{proof}
	The proof of the above theorem is based on the universality of signature kernel $k_S$ (see Appendix \ref{sect: kernel learning}) and the analytical property of signature mapping (see sect. \ref{subsect: analytic signature}). See the proof of \cite[Theorem 7]{Salvi2021HighRank} for all details.
\end{proof}	
Thanks to the above theorem, we can approximate $\mathcal{D}_{\mathcal{S}}^2(\mathbb{X}, \mathbb{Y}) = \|\mu^2_{\hat X^1} - \mu^2_{\hat Y^1}\|_{\mathcal{H}^2_{\mathcal{S}}}$ arbitrarily well by computing their empirical counterparts $\|\hat \mu^2_{\hat X^1} - \hat \mu^2_{\hat Y^1}\|_{\mathcal{H}^2_{\mathcal{S}}}$ with $$\hat \mu^2_{\hat X^1} := \mathbb{E}_{\mathbb{P}_X}[k_{S_{\mathcal{H}_{\mathcal{S}}}}(t \mapsto \hat \mu_{\hat X^1_t}((\bar x^i)_{i \le M},\cdot))] $$
and $\hat \mu^2_{\hat Y^1} := \mathbb{E}_{\mathbb{Q}_Y}[k_{S_{\mathcal{H}_{\mathcal{S}}}}(t \mapsto \hat \mu_{\hat Y^1_t}((\bar y^j)_{j \le N},\cdot))] $ by using  $(\bar x^i)_{i \le M} \sim \mathbb{P}_X$ and $(\bar y^j)_{j \le N} \sim \mathbb{Q}_Y$ for large enough $M$ and $N$. Now we fix samples $(\bar x^i)_{i \le M} \sim \mathbb{P}_X$ and $(\bar y^j)_{j \le N} \sim \mathbb{Q}_Y$ and, for simplicity, write $t \mapsto \hat \mu_{\hat X^1_t}((\bar x^i)_{i \le M},\cdot) = \hat \mu_{\hat X^1}$ and $t \mapsto \hat \mu_{\hat Y^1}((\bar y^j)_{j \le N},\cdot) = \hat \mu_{\hat Y^1}$. Using the polarization expansion as  in the proof of Lemma \ref{lemma: MMD by RKHS} to $\|\hat \mu^2_{\hat X^1} - \hat \mu^2_{\hat Y^1}\|_{\mathcal{H}^2_{\mathcal{S}}}^2$, and then approximating all terms appeared in the resulting expansion by using empirical measures $\frac{1}{m}\sum_{i=1}^m \delta_{x^i}$ for $\mathbb{P}_X$ with $(x^i)_{i \le m} \sim \mathbb{P}_X$ and $\frac{1}{n}\sum_{j=1}^n \delta_{y^j}$ for $\mathbb{Q}_Y$ with $(y^j)_{j \le n} \sim \mathbb{Q}_Y$, we finally derive an unbiased empirical estimator for $\|\hat \mu^2_{\hat X^1} - \hat \mu^2_{\hat Y^1}\|_{\mathcal{H}^2_{\mathcal{S}}}^2$:
\begin{align}\label{eq: estimator for 2nd MMD}
	\widehat{\mathcal{D}}^2_{\mathcal{S}}(\mathbb{X},\mathbb{Y})^2 &= \frac{1}{m(m-1)} \sum_{\substack{i,j=1 \\ i\neq j}}^m k_{S_{\mathcal{H}_{\mathcal{S}}}}(\widetilde{x}^i,\widetilde{x}^j) - \frac{2}{mn} \sum_{i,j=1}^{m,n} k_{S_{\mathcal{H}_{\mathcal{S}}}}(\widetilde{x}^i,\widetilde{y}^j) \nonumber \\
	&\quad + \frac{1}{n(n-1)} \sum_{\substack{i,j=1 \\ i\neq j}}^n k_{S_{\mathcal{H}_{\mathcal{S}}}}(\widetilde{y}^i,\widetilde{y}^j)
\end{align}
where $\tilde x^i_t := \hat \mu_{\hat X^1_t}(x^i|_{I_t}) = \widehat{F_{t,k_S}}((\bar x^k)_{k \le M},x^i|_{I_t})$, $t \in I$ and 
$\tilde y^j_t := \hat \mu_{\hat Y^1_t}(y^j|_{I_t}) = \widehat{F_{t,k_S}}((\bar y^k)_{k \le N},y^j|_{I_t})$, $t \in I$. Since $\frac{1}{m}\sum_{i=1}^m \delta_{x^i} \to \mathbb{P}_X$, $\frac{1}{n}\sum_{j=1}^n \delta_{y^j} \to \mathbb{Q}_Y$ weakly as $m,n \to \infty$ and the empirical estimators $\widehat{F_{t,k_S}}$ are continuous, we indeed obtain that
$$
\widehat{\mathcal{D}}^2_{\mathcal{S}}(\mathbb{X},\mathbb{Y})^2 \to \|\hat \mu^2_{\hat X^1} - \hat \mu^2_{\hat Y^1}\|_{\mathcal{H}^2_{\mathcal{S}}}^2 \approx \mathcal{D}_{\mathcal{S}}^2(\mathbb{X}, \mathbb{Y})
$$
as $m,n \to \infty$. This observation together with Theorem \ref{thm: convergence of empirical estimates} immediately implies the desired convergence behaviour of the empirical $2$--nd MMD $\widehat{\mathcal{D}}^2_{\mathcal{S}}(\mathbb{X},\mathbb{Y})$ to the true metric $\mathcal{D}_{\mathcal{S}}^2(\mathbb{X}, \mathbb{Y})$:
\begin{theorem}\label{thm: convergence of empirical 2nd MMD}
	For any $\mathbb{X}, \mathbb{Y} \in \Lambda_{\text{plain}}(\mathcal{K})$, $\widehat{\mathcal{D}}^2_{\mathcal{S}}(\mathbb{X},\mathbb{Y})$ is a consistent estimator for the $2$--nd MMD, i.e.,
	$$
	\widehat{\mathcal{D}}^2_{\mathcal{S}}(\mathbb{X},\mathbb{Y}) \to \mathcal{D}_{\mathcal{S}}^2(\mathbb{X}, \mathbb{Y}) 
	$$
	in probability as $m, n \to \infty$ and $M,N \to \infty$, provided the parameters $\lambda_M$, $\lambda_N$ appeared in their formulations decay to $0$ slower than $\mathcal{O}(M^{-1/2})$ and $\mathcal{O}(N^{-1/2})$ respectively.
\end{theorem}
	
\subsection{A signature kernel PDE Trick}\label{sect: PDE trick}
Thanks to Theorem \ref{thm: convergence of empirical 2nd MMD}, we can approximate $\mathcal{D}_{\mathcal{S}}^2(\mathbb{X}, \mathbb{Y})$ by computing its empirical counterpart $\widehat{\mathcal{D}}^2_{\mathcal{S}}(\mathbb{X},\mathbb{Y})$. Then from its expression \eqref{eq: estimator for 2nd MMD}, computing this estimator boils down to evaluating the signature kernels $k_{S_{\mathcal{H}_{\mathcal{S}}}}(\widetilde{x},\widetilde{y})$ for $\tilde x = \tilde x^i$ and $\tilde y = \tilde y^j$ for some $i,j$. By \cite[Theorem 2.5]{cass2020computing}  the signature kernel $k_{S_{\mathcal{H}_{\mathcal{S}}}}(\widetilde{x},\widetilde{y})$ is the unique solution $u(T,T)$ of the following hyperbolic PDE, which is also called the Goursat problem:
\begin{equation*}
	\frac{\partial^2 u_{\widetilde{x},\widetilde{y}}}{\partial s \partial t}  = \big\langle \frac{\partial}{\partial t}\widetilde{x}_t, \frac{\partial}{\partial s}\widetilde{y}_{s}\big\rangle_{\mathcal{H}_{\mathcal{S}}} u_{\widetilde{x},\widetilde{y}}
\end{equation*}
where we consider $\tilde x$ and $\tilde y$ as continuous paths in $\mathcal{H}_{\mathcal{S}}$ via linear interpolation (see Remark \ref{remark: continuous time signature}). Since the two derivatives in the above equation can be approximated by finite difference with time increment $\delta$, one can approximate the solution to the above Goursat problem by solving its discretized version:	
\begin{equation}\label{eq: goursat pde for higher order kernel}
	\frac{\partial^2 u_{\widetilde{x},\widetilde{y}}}{\partial s \partial t}  =\frac{1}{\delta^2} \left(\big\langle \widetilde{x}_{t-\delta},\widetilde{y}_{s-\delta}\big\rangle_{\mathcal{H}_{\mathcal{S}}} - \big\langle \widetilde{x}_{t-\delta},\widetilde{y}_{s}\big\rangle_{\mathcal{H}_{\mathcal{S}}} - \big\langle \widetilde{x}_{t},\widetilde{y}_{s-\delta}\big\rangle_{\mathcal{H}_{\mathcal{S}}} + \big\langle \widetilde{x}_{t},\widetilde{y}_{s}\big\rangle_{\mathcal{H}_{\mathcal{S}}} \right) u_{\widetilde{x},\widetilde{y}}
\end{equation}	
Now invoking that $\tilde x_t := \hat \mu_{\hat X^1_t}(x|_{I_t}) = \widehat{F_{t,k_S}}((\bar x^k)_{k \le M},x|_{I_t})$, $t \in I$ and 
$\tilde y_s := \hat \mu_{\hat Y^1_s}(y|_{I_s}) = \widehat{F_{s,k_S}}((\bar y^k)_{k \le N},y|_{I_s})$, $s \in I$, by the expression \eqref{eq: minimizer for emirical loss} for $\widehat{F_{t,k_S}}$, we can readily check that $\forall s,t \in I$
\begin{equation}\label{eq: empirical es for inner product}
	\big\langle \widetilde{x}_{t},\widetilde{y}_{s}\big\rangle_{\mathcal{H}_{\mathcal{S}}}   = \textbf{k}^x_s  (\textbf{K}^{X}_{t,t} + M \lambda_M \text{I}_M)^{-1}\textbf{K}^{X,Y}_{T,T}(\textbf{K}^{Y}_{s,s} + N \lambda_N \text{I}_N)^{-1}\textbf{k}^y_s,
\end{equation}
where $\textbf{k}^x_t \in \mathbb{R}^M$ and $\textbf{k}_t^s \in \mathbb{R}^N$ are the vectors
\begin{equation}\label{eq: empirical vector}
	[\textbf{k}^x_t]_i = k_{S_t}(\bar x^i|_{I_t}, x|_{I_t}), \quad [\textbf{k}^y_s]_j = k_{S_s}(\bar y^j|_{I_s}, y|_{I_s}), \quad i \le M, j\le N
\end{equation}
and $\textbf{K}^X_{s,s} \in \mathbb{R}^{M \times M}$, $\textbf{K}^{X,Y}_{T,T} \in \mathbb{R}^{M \times N}$ and $\textbf{K}^Y_{t,t}\in \mathbb{R}^{N \times N}$ are matrices
\begin{align}\label{eq: empirical matrix}
	&[\textbf{K}^X_{t,t}]_{i,j} = k_{S}(\bar x^i|_{I_t}, \bar x^j|_{I_t}), \quad  [\textbf{K}^{X,Y}_{T,T}]_{i,j} = k_S(\bar x^i,\bar y^j),  \nonumber\\
	&[\textbf{K}^Y_{s,s}]_{i,j} = k_S(\bar y^i|_{I_s}, \bar y^j|_{I_s}).
\end{align}
Now, by using the signature kernel PDE trick \cite[Theorem 2.5]{cass2020computing} again to these kernels which are related to the normal signature map $S$ on $\mathcal{K}^I$, namely $k_S(\bar x^i,\bar y^j)$ is the unique solution $u(T,T)$ to the Goursat problem
$$
\frac{\partial^2 u_{\bar x^i,\bar y^j}}{\partial s \partial t}  = \big\langle \frac{\partial}{\partial t}\bar x^i_t, \frac{\partial}{\partial s}\bar y^j_{s}\big\rangle_{\mathcal{H}_{\mathcal{S}}} u_{\bar x,\bar y}
$$
(and similar expressions hold for $k_{S}(\bar x^i|_{I_t}, \bar x^j|_{I_t})$ and $k_S(\bar y^i|_{I_s}, \bar y^j|_{I_s})$), we can compute the above vectors and matrices, and consequently obtain a numerical solution to $\langle \tilde x^i_t, \tilde{y}^j_s \rangle_{\mathcal{H}_{\mathcal{S}}}$ in \eqref{eq: empirical es for inner product} for all $i,j$. Then, inserting them back to the coefficients of PDE \eqref{eq: goursat pde for higher order kernel} we can compute the higher order kernel $k_{S_{\mathcal{H}_{\mathcal{S}}}}(\tilde x^i,\tilde y^j)$ and finally obtain the desired numerical result for the empirical estimator $\widehat{\mathcal{D}}^2_{\mathcal{S}}(\mathbb{X},\mathbb{Y})$ via the formula \eqref{eq: estimator for 2nd MMD} by using sample paths $\tilde x^i$ and $\tilde y^j$, $i \le m$ and $j\le n$.

\subsection{Overview of the algorithm}\label{sect: summary}
Now we summarize our algorithm for computing the estimator $\widehat{\mathcal{D}}^2_{\mathcal{S}}(\mathbb{X}, \mathbb{Y})$ (where $\mathbb{X}, \mathbb{Y} \in \Lambda_{\text{plain}}(\mathcal{K})$) into the following diagram. To ease our notation, let us assume that (in the context of Theorem \ref{thm: convergence of empirical 2nd MMD}) $M = m$, $N=n$, $\bar x^i = x^i$, $i = 1, \ldots, m$ and $\bar y^j = y^j$, $j = 1, \ldots, n$ are samples from $\mathbb{P}_X$ and $\mathbb{Q}_Y$, respectively. Also recall that $\mathbb{X}$ and $\mathbb{Y}$ are defined on the discrete time interval $I = \{0=t_0 < t_1 < \ldots < t_N= T\}$ for some integer $N \ge 1$.

\begin{algorithm}
	\caption{$\mathsf{PDESolve}$ \hfill $\mathcal{O}(N^2)$}
	\label{alg:pdesolve}
	\begin{algorithmic}[1]
		\State {\bfseries Input:} matrix $M\in\mathbb{R}^{N\times N}$ where $[M]_{p,q} = \langle x^i_{s_{p+1}} - x^i_{s_p}, y^j_{t_{q+1}} - y^j_{t_q} \rangle_{\mathbb{R}^d}$, $\mathsf{full}\in\{\mathsf{True},\mathsf{False}\}$
		\State {\bfseries Output:} full solution $u\in\mathbb{R}^{N\times N}$ with $u[p,q]=k_{S}(x|_{I_{t_p}},y|_{I_{t_q}})$ or $u[-1,-1]=k_{\mathcal{S}}(x,y)$
		\vspace{5pt}
		\State{$u[1,:]\leftarrow 1$}
		\State{$u[:,1]\leftarrow 1$}
		\For{$p$ from $1$ to $N-1$}
		\For{$q$ from $1$ to $N-1$}
		\State $u[p+1,q+1]\leftarrow f (u[p,q+1], u[p+1,q], u[p,q], M[p,q] )$
		\EndFor
		\EndFor
		\State {\bfseries if} $\mathsf{full}$ {\bfseries then return} $u$ {\bfseries else}  {\bfseries return} $u[-1,-1]$
	\end{algorithmic}
\end{algorithm}

\begin{algorithm}
	\caption{$\mathsf{FirstOrderGram}$ \hfill $\mathcal{O}(dm^2N^2)$}
	\label{alg:firstordergram}
	\begin{algorithmic}[1]
		\State {\bfseries Input:} sample paths $\{x^i\}_{i=1}^{m}\sim \mathbb{P}_X$ and $\{y^j\}_{j=1}^{n}\sim \mathbb{Q}_Y$, $\mathsf{full}\in\{\text{True},\text{False}\}$
		\State {\bfseries Output:} $G\in\mathbb{R}^{m\times n\times P\times Q}$ where $G[i,j,p,q]=k_{S}(x^i|_{I_{s_p}},y^j|_{I_{t_q}})$ or $G[:,:,-1,-1]$
		\vspace{5pt}
		\State $M[i,j,p,q] \leftarrow \langle x^i_{s_p}, y^j_{t_q}\rangle\quad \forall i\in\{1,\ldots,m\},~ j\in\{1,\ldots,n\}, p\in\{1,\ldots,P\},~ q\in\{1,\ldots,Q\}$ 
		\State $M\leftarrow M[:,:,1{:},1{:}] + M[:,:,{:}\!-\!1,{:}\!-\!1] - M[:,:,1{:},{:}\!-\!1] - M[:,:,{:}\!-\!1,1{:}]$ 
		\State $G[i,j]\leftarrow \mathsf{PDESolve}(M[i,j]),\quad \forall i\in\{1,\ldots,m\},~ j\in\{1,\ldots,n\}$ 
		\State {\bfseries if} $\mathsf{full}$ {\bfseries then return} $G$ {\bfseries else}  {\bfseries return} $G[:,:,-1,-1]$
	\end{algorithmic}
\end{algorithm}

\begin{algorithm}[H]
	\caption{$\mathsf{InnerProdPredCondKME}$ \hfill $\mathcal{O}(N^2m^3)$}
	\label{alg:innerprodpredcondKME}
	\begin{algorithmic}[1]
		\State {\bfseries Input:} three Gram matrices $G_{XX}, G_{XY}, G_{YY}$ and hyperparameter $\lambda$, where $G_{XY}[i,j,p,q]=k_{S}(x^i|_{I_{s_p}},y^j|_{I_{t_q}})$ and similar expressions for $G_{XX}$ and $G_{YY}$.
		\State {\bfseries Output:} returns an empirical estimator of $M\in\mathbb{R}^{m\times n\times N\times N}$ where $M[i,j,p,q]=\langle\widetilde{x}^i_{s_p},\widetilde{y}^j_{t_q}\rangle$
		\State $W_X[:,:,p]\leftarrow (G_{XX}[:,:,p,p]+m\lambda I)^{-1}, \ \ \forall p\in\{1,\ldots,P\}$ 
		\State $W_Y[:,:,q]\leftarrow (G_{YY}[:,:,q,q]+n\lambda I)^{-1}, \ \ \forall q\in\{1,\ldots,Q\}$ 
		\For{$p$ from $1$ to $P$}
		\For{$q$ from $1$ to $Q$}
		\State $M[:,:,p,q]\leftarrow G_{XX}[:,:,p,p]^T W_X[:,:,p] G_{XY}[:,:,-1,-1] W_Y[:,:,q] G_{YY}[:,:,q,q]$
		\EndFor
		\EndFor
	\end{algorithmic}
\end{algorithm}

\begin{algorithm}
	\caption{$\mathsf{SecondOrderGram}$ \hfill $\mathcal{O}(N^2m^3)$}
	\label{alg:secondordergram}
	\begin{algorithmic}[1]
		\State {\bfseries Input:} $G_{XX}, G_{XY}, G_{YY}$  and hyperparameter $\lambda$.
		\State {\bfseries Output:} an empirical estimator of $G^2_{X,Y}\in\mathbb{R}^{m\times n}$, where $G^2_{X,Y}[i,j]=k_{S_{\mathcal{H}_{\mathcal{S}}}}(\widetilde{x}^i,\widetilde{y}^j)$
		\vspace{5pt}
		\State $M \leftarrow\mathsf{InnerProdPredCondKME}(G_{XX},G_{XY},G_{YY},\lambda)$
		\State $M\leftarrow M[:,:,1{:},1{:}] + M[:,:,{:}\!-\!1,{:}\!-\!1] - M[:,:,1{:},{:}\!-\!1] - M[:,:,{:}\!-\!1,1{:}]$
		\State $G^2_{XY}[i,j]\leftarrow \mathsf{PDESolve}(M[i,j]),\quad \forall i\in\{1,\ldots,m\},~\forall j\in\{1,\ldots,n\}$ 
		\State {\bfseries return} $G^2_{XY}$
	\end{algorithmic}
\end{algorithm}

\begin{algorithm}
\caption{$\mathsf{SecondOrderMMD}$  \hfill $\mathcal{O}(dm^2N^2+N^2m^3)$}
\label{alg:secondordermmd}
\begin{algorithmic}[1]
	\State {\bfseries Input:} sample paths $\{x^i\}_{i=1}^{m}\sim \mathbb{P}_X$ and $\{y^j\}_{j=1}^{n}\sim \mathbb{Q}_Y$, hyperparameter $\lambda$.
	\State {\bfseries Output:} an empirical estimator of the $2^{\text{nd}}$ order MMD $\widehat{\mathcal{D}}^2_{\mathcal{S}}(\mathbb{X}, \mathbb{Y})$.
	\vspace{5pt}
	\State $G^1_{XX}\leftarrow \mathsf{FirstOrderGram}(\{x^i\}_{i=1}^{m},\{x^i\}_{i=1}^{m},\mathsf{full}=\mathrm{True})$
	\State $G^1_{XY}\leftarrow \mathsf{FirstOrderGram}(\{x^i\}_{i=1}^{m},\{y^j\}_{j=1}^{n},\mathsf{full}=\mathrm{True})$
	\State $G^1_{YY}\leftarrow \mathsf{FirstOrderGram}(\{y^j\}_{j=1}^{n},\{y^j\}_{j=1}^{n},\mathsf{full}=\mathrm{True})$
	\vspace{5pt}
	\State $G^2_{XX}\leftarrow \mathsf{SecondOrderGram}(G^1_{XX},G^1_{XX},G^1_{XX},\lambda)$
	\State $G^2_{XY}\leftarrow \mathsf{SecondOrderGram}(G^1_{XX},G^1_{XY},G^1_{YY},\lambda)$
	\State $G^2_{YY}\leftarrow \mathsf{SecondOrderGram}(G^1_{YY},G^1_{YY},G^1_{YY},\lambda)$
	\vspace{5pt}
	\State {\bfseries return} $\mathrm{avg}(G^2_{XX}) -2*\mathrm{avg}(G^2_{XY}) + \mathrm{avg}(G^2_{YY})$
\end{algorithmic}
\end{algorithm}

Finally, in view of Theorem \ref{thm: main theorem 3}, we obtain an consistent estimator $\hat K^2_{\mathcal{S}}: = \exp(-\sigma^2\widehat{\mathcal{D}}^2_{\mathcal{S}})$ for the ($2$--nd order)  Distribution Regression kernel $K^2_{\mathcal{S}}$. As a consequence, given a dataset $\mathcal{D} = \{(\mathbb{Y}^i, v_i)\}_{i=1}^M$ of input--output pairs, where $\mathbb{Y}^i \in \Lambda_{\text{plain}}(\mathcal{K})$ are financial models and $v_i$ is the value of $v(\mathbb{Y}^i) = \sup_{\tau \in \mathcal{T}_{\mathbb{Y}^i}}\mathbb{E}_{\mathbb{Q}^i}[\gamma(Y^i_\tau)]$  (in the context of mathematical finance, it is the fair price of American option on the market modelled by $\mathbb{Y}^i$), we can perform the Kernel Ridge Regression via the $2$--nd order kernel $\hat K^2_{\mathcal{S}}$. More precisely, the approximating value function $\hat v(\cdot)$ on $\Lambda_{\text{plain}}(\mathcal{K})$ admits the following form:
\begin{equation}\label{eq: kernel ridge regression}
	\hat v (\cdot) = \sum_{i=1}^M a_i \hat K^2_{\mathcal{S}}(\cdot, \mathbb{Y}^i), 
\end{equation}
where $\textbf{a} = (a_i)_{i\le M} = (\widehat{\textbf{K}}^2_{\mathcal{S}} + \lambda \text{I}_M)^{-1}\textbf{v}$, $[\widehat{\textbf{K}}^2_{\mathcal{S}}]_{ij} = \hat K^2_{\mathcal{S}}(\mathbb{Y}^i, \mathbb{Y}^j)$. Finally, to evaluate $\hat v (\cdot)$ at a new model $\mathbb{X}$, we need to compute the $M$--dimensional vector $(\hat K^2_{\mathcal{S}}(\mathbb{X}, \mathbb{Y}^i))_{i \le M}$ by using the above algorithm.

\subsection{Numerical examples}\label{sect: examples}
In this section we empirically investigate the effectiveness of the proposed signature kernel algorithm for two optimal stopping problems:
	\begin{enumerate}
		\item the pricing of American basket option with geometric put payoff;
		\item the optimal stopping of fractional Brownian motion.
	\end{enumerate}
	More precisely, treating these as distribution regression problems from probability measures on pathspace to $\mathbb{R}$, we endow the classical kernel Ridge regression (KRR) algorithm \cite{cristianini2000introduction} with the characteristic kernel $K^2_{\mathcal{S}}$ studied in previous sections and compare its performance against classical algorithms. All experiments have been run on an NVIDIA Tesla P100 GPU.

\subsubsection{American basket option pricing}

An American option gives the holder the right but not the obligation to exercise the option
associated with a non-negative payoff function $\Phi$ at any time up to maturity. An American
option can be approximated by a Bermudan option, which can be exercised only at 
specific dates $t_0 < t_1 < t_2 < ... < t_N$. If the time grid is chosen small enough, the American option is well approximated by the
Bermudan option. For $d \in \mathbb{N}$, we consider the Black-Scholes model by considering a $d$-dimensional geometric Brownian motion $(X_t)_{t\geq0}$ for the stock prices. At time $t \in [0,T]$ and for initial stock prices $X_t = x \in \mathbb{R}^d$, the price of the American option $V(t,x)$ can then be formulated as the following optimal stopping problem
\begin{equation*}
	V(t,x) =  \sup_{\tau \in [t,T]} \mathbb{E}^{\mathbb{Q}}\biggl[ e^{-r(\tau-t)} \Phi(S_{\tau}) \bigg| S_t=x \biggr]
\end{equation*}
with payoff function $\Phi : \mathbb{R}^d \to \mathbb{R}_+$, and where the supremum is taken over all stopping times. Here, we consider a geometric put payoff defined as follows
$$\Phi(x) = \left(K - \left(\prod_{i=1}^d x_i\right)^{1/d} \right)_+$$
with strike $K > 0$.

\begin{remark}
	As explained in \cite{lapeyre2021neural}, geometric put options on $d$-dimensional stocks following Black–Scholes are equivalent to one-dimensional put options on a $1$-dimensional stock following Black–Scholes with adjusted parameters. This observation provides a way to derive a "ground-truth" derivative price via the CRR binomial-tree method \cite{cox1979option}, in the sense that the price computed with this method converges to the correct price under the Black–Scholes model as the depth of the tree goes to infinity. This ground-truth answer can be then used to assess performance of alternative algorithms, as we shall do below.
\end{remark}

To perform KRR, we consider the classical supervised learning setting and we form a dataset $\{(\mathcal{X}_i, y_i)\}_{i=1}^N$ of $N = 100$ input-output pairs, where: 
\begin{itemize}
	\item $\mathcal{X}_i = \{x^{i,j} : [0,T] \to \mathbb{R}^d\}_{j=1}^{n_i}$ is a set of $n_i$  paths sampled from a $d$-dimensional Black-Scholes model, with fixed interest rate $r = 2 \%$, $d=20$ and volatility sampled uniformly at random $\sigma_i \in [0.1,0.5]$ over the time grid $\{0, T/\ell, 2T/\ell, ..., T\}$ with $\ell= 10$ as in \cite{herrera2021optimal}. For simplicity we set $n_1 = .... = n_N = n$. 
	\item $y_i$ is the price of the American option obtained via a binonal-tree method using a tree of depth $10\,000$ with strike $K=100$ and maturity $T=1$.
\end{itemize}

The dataset is randomly split into a training set of size $N_{\text{train}}=90$ and a test set of size $N_{\text{test}}=10$. We fit the KRR on the training set and evaluate the model on the test set.

\begin{remark}
	We note that here we are using sample paths over a coarse discretisation as inputs to our model, while the target labels $\{y_i\}_{i=1}^N$ are obtained using a much finer discretisation, which can be assumed to have converged to the true American option price. This means that our algorithm is learning the Bermudan option price as an approximation to the true American option price. 
\end{remark}

We benchmark our supervised learning kernel approach against the state-of-the art Longstaff and Schwartz (LS) algorithm introduced in \cite{longstaff2001valuing} by counting how many sample paths $n$ are required by both methods to produce an approximation error $\epsilon$ from the the target price obtained using the lattice method. We use the Mean Absolute Percent Error (MAPE), that is $\epsilon=N_{\text{test}}^{-1}\sum_{i=1}^{N_{\text{test}}}{|(\hat{y}_i-y_i)/y_i|}$ where $\{\hat{y}_i\}_{i=1}^{N_{\text{test}}}$ are the prices predicted by either algorithm. 
We use the code provided in \cite{herrera2021optimal} to apply the LS algorithm with basis functions given by the Taylor polynomials
up to degree two, as recommended in \cite{herrera2021optimal}.
As shown in Table \ref{table:basket_option}, the higher rank signature methods, that is KRR with kernels $K^2_{\mathcal{S}}$ and $K^1_{\mathcal{S}}$, both require significantly fewer sample paths from the models to produce a desired precision. We observe that for higher rank signature methods, the regression is fitted on $N_{\text{train}}$ base models and so depends on $n\times N_{\text{train}}$ sample paths. However, for any new prediction, the $n\times N_{\text{train}}$ training sample paths are fixed and do not need to be re-sampled; only $n$ additional samples are required. So for a total of $N_{\text{test}}$  new predictions we will have that $n(N_{\text{train}} + N_{\text{test}}) \approx nN_{\text{test}}$ when $N_{\text{train}} << N_{\text{test}}$.

\begin{table}[h]
	\caption{Geometric put option on Black–Scholes. Number of samples $n$ to produce a MAPE $\epsilon$ on the test set. }\label{table:basket_option}
	\begin{center}
		\begin{tabular}{l c c c } 
			\toprule
			& $\epsilon=10\%$ &$\epsilon=5\%$ & $\epsilon=2.5\%$\\ 
			\midrule
			Longstaff and Schwartz  &  8\,000 & 16\,000 &60\,000\\
			Higher rank signature (rank 1) &  50 & 250 & 1\,000 \\
			Higher rank signature (rank 2) & 50 & 200 & 500\\
			
	\bottomrule	
			
		\end{tabular}%
		
	\end{center}
\end{table}

\subsubsection{Stopped fractional Brownian motion}

To showcase the flexibility of our approach, we consider a second, this time non-Markovian, example of stopped fractional Brownian motion following \cite{Bayer2021stopping}. More precisely, we consider a one dimensional fractional Brownian motion $X^H$ with Hurst exponent $H \in [0,1)$ and tackle the optimal stopping problem where the payoff is given by the underlying process itself. 

To perform this optimal stopping task via kernel Ridge regression, as done in the previous example, we form a dataset $\{(\mathcal{X}_i, y_i)\}_{i=1}^N$ of input-output pairs, where: 
\begin{itemize}
	\item $\mathcal{X}_i = \{x^{i,j} : [0,1] \to \mathbb{R}^2\}_{j=1}^{n_i}$ is a set of $n_i$  paths sampled from a time-augmented fractional Brownian motion with Hurst exponent $H_i = (i+1)/40$ for $i=1,\ldots, 39$ over the time grid $\{0, 1/\ell, 2/\ell, ..., 1\}$ with $\ell = 10$. For simplicity we set $n_1 = .... = n_N = n = 500$;
	\item $\{y_i\}_{i=1}^{N}$ are the values obtained in \cite{becker2019deep}, with $\ell=100$.
\end{itemize}
We take $H = i/10$ for $i=1,\ldots, 10$ as a test set and all the others as training set. As a mean of assessing the performance of our model, in the table below we report the predictions obtained using the higher rank signature methods (first two columns) and compare them with the lower and upper bound estimates on the optimal value obtained in \cite{becker2019deep} (last two columns). As it can be observed, our predictions fall close to the target interval.

\begin{table}[h]
	\caption{Optimal stopping of fractional Brownian motion.}
	\begin{center}
		\begin{tabular}{l c c c c} 
			\toprule
			\multirowcell{2}[0pt][l]{} \multirowcell{2}[0pt][l]{H} &  \multicolumn{2}{c}{Higher rank signatures}  & \multicolumn{2}{c}{\cite{becker2019deep}}\\  
			& rank $1$ & rank $2$ &  lower bound & upper bound  \\ \midrule
			
			0.1  & 1.085  & \textbf{1.061} &  1.048 & 1.049\\
			
			0.2 & \textbf{0.684} &  0.689 & 0.658 & 0.659 \\
			0.3 & 0.363 &  \textbf{0.376} & 0.369 & 0.380 \\
			0.4 & 0.137 & \textbf{0.174}  & 0.155 & 0.158 \\
			0.5 & 0.070 & \textbf{0.065} &  0.000 & 0.005 \\
			0.6 & 0.177 & \textbf{0.143} &  0.115 & 0.118 \\
			0.7 & \textbf{0.189} & 0.229  & 0.206 & 0.207 \\
			0.8 & \textbf{0.278} & \textbf{0.276}  & 0.276 & 0.278 \\
			0.9 & 0.324 & \textbf{0.330} & 0.336 & 0.339 \\ 
			1.0 & 0.375 & \textbf{0.390}  & 0.395 & 0.395 \\
			\bottomrule
		\end{tabular}%
		\label{table:experiment}
	\end{center}
\end{table}

\section{Summary and Future Work}\label{sect: summary and future work}
In this section let us briefly summarize the main results obtained in the present paper and the essential ideas hidden behind them.
	\begin{itemize}
		\item The signature kernel regression approach can be used for calculating the values of pricing functionals of path--dependent derivatives which are weakly continuous. However, since the value functions of Optimal Stopping Problems are in general not continuous for the weak topology, this classical approach based on the normal expected signature cannot remain valid for the numerical computation of the prices for American type options.
		\item To overcome this discontinuity issue, D. Aldous introduced the notion of extended weak topology (also called the rank $1$ adapted topology in this paper) which is the weak topology induced by the prediction process $\hat X_t = \mathbb{P}[X \in \cdot| \mathcal{F}_t]$, $t \in I$ of the triple $(X, \mathbb{F}, \mathbb{P})$. Moreover, Aldous proved that if all filtrations are natural, i.e., generated by the processes themselves, then the value functions of OSPs (attached to continuous bounded payoffs) are continuous with respect to the extended weak topology.
		\item The rank $2$ signature $S^2$ can be viewed as the signature of the measure--valued path $\hat X_t$, $t \in I$. For a prediction process $\hat X$, $S^2(\hat X)$ takes values in a Hilbert space $\mathcal{H}^2$ which is embedded in the tensor algebra of $\mathcal{H}^1$, where the latter is the target space of normal signatures $S$ for Euclidean--valued paths. Using the techniques from kernel learning, we showed that  the $\mathcal{H}^2$--norm distances between expected rank $2$ signatures induce a reproducing kernel $K^2_{\mathcal{S}}(\cdot, \cdot)$ whose associated RKHS is dense in the space of bounded and extended weakly continuous functions on the set of filtered processes with natural filtrations. Hence, the value functions of OSPs can be approximated by linear combinations of kernel $K^2_{\mathcal{S}}(\cdot, \cdot)$ arbitrarily well, just like approximating continuous functions by polynomials. 
		\item One main contribution of this paper is that we exploit the CKME (Conditional Kernel Mean Embedding) approach to give an empirical estimation for the kernel $K^2_{\mathcal{S}}(\cdot, \cdot)$ so that one can find numerical solutions to the prices of American options by performing a kernel regression. Furthermore, thanks to the intrinsic dynamics satisfying by the rank $2$ signatures and the inner product structure of $\mathcal{H}^2$, we showed that our numerical scheme for the computation of kernel $K^2_{\mathcal{S}}(\cdot, \cdot)$ essentially only involves with solving a family of $2$--dimensional linear hyperbolic PDEs (the Goursat problem), whose coefficients consist of  Euclidean inner product of $\mathbb{R}^d$--valued vectors. This signature PDE trick helps us to reduce the complexity of this infinite dimensional regression problem significantly.
	\end{itemize}
	For the relevant future work \footnote{We appreciate the comments and questions from anonymous referees which motivate us for studying these topics mentioned here.} we can ask the following interesting questions:
	\begin{itemize}
		\item How the choices of base models $(\mathbb{Y}^i)_{i \le M}$ used in kernel regression affect the performance of the algorithm. 
		\item Applying the approach presented in this paper to more complicated exotic options of American type; for instance, the American moving average options as in \cite{Bernhart2011movingaverage}, whose value functions depend on the history of the paths rather than the spot value only. As mentioned at the end of the introduction, in principle our procedure works for pricing of such options provided we could show that their value functions are bounded and continuous with respect to the extended weak topology or adapted Wasserstein topology.
		\item As we have shown in Appendix \ref{subsect: discretize}, given a model on continuous time, then theoretically the value functions on discretized markets converge to the counterpart of the original continuous time model for any time--discretization sequences with mesh size tending to $0$. It is interesting to investigate how the time--discretization scheme affects the corresponding convergence rate in our algorithm. We expect to get a quantitative estimates by using the tools developed in \cite{Bartl2021Wasserstein}.
		\item All the theoretical results presented in this paper are based on the Stone--Weierstrass theorem, which requires the target functions to be bounded and continuous. However, many reward functions in stochastic optimal control  problems are unbounded. As a first step to generalize our approach to unbounded cases, we may expect to show that the (higher rank) expected signatures induce a RKHS which is dense in $L^p$ space for those reward functions of growth order $p \ge 1$, as in the recent paper \cite{Bartl2021Wasserstein}. 
\end{itemize}

\appendix

\section{Basic Properties of Signatures}\label{appendix: signature}

\subsection{Hilbert space structure on tensor algebra}\label{subsect: hilbert space on tensor algebra}
Let $(\mathcal{H}, \langle \cdot, \cdot \rangle_{\mathcal{H}})$ be a separable Hilbert space, and we fix an orthonormal basis $(e_i)_{i \in \mathbb{N}}$ for $\mathcal{H}$. For each $n \ge 1$ we consider the $n$--th tensor product of $\mathcal{H}$, namely $\mathcal{H}^{\otimes n} = \text{span}(h_1 \otimes \ldots \otimes h_n: h_i \in \mathcal{H})$. We endow $\mathcal{H}^{\otimes n}$ with an inner product $\langle \cdot,\cdot \rangle_{\mathcal{H}^{\otimes n}}$ such that 
$$
\langle h_1 \otimes \ldots \otimes h_n,  h^\prime_1 \otimes \ldots \otimes h^\prime_n \rangle_{\mathcal{H}^{\otimes n}} := \prod_{i=1}^n \langle h_i, h^\prime_i \rangle_{\mathcal{H}}.
$$
Let $\mathcal{H}^{\hat \otimes n}$ be the completion of $\mathcal{H}^{\otimes}$ with respect to $\langle \cdot,\cdot \rangle_{\mathcal{H}^{\otimes n}}$, which is a separable Hilbert space with $\{e_{i_1} \otimes \ldots \otimes e_{i_n}: i_1,\ldots, i_n \in \mathbb{N}\}$ being an orthonormal basis. Now we define the following spaces:
\begin{itemize}
	\item $T(\mathcal{H}) = \bigoplus_{n=0}^\infty \mathcal{H}^{\otimes n}$,
	\item $\textbf{T}(\mathcal{H}) = \bigoplus_{n=0}^\infty \mathcal{H}^{\hat \otimes n}$,
	\item $T((\mathcal{H})) = \prod_{n=0}^\infty \mathcal{H}^{\otimes n}$, 
	\item $\textbf{T}((\mathcal{H})) = \prod_{n=0}^\infty \mathcal{H}^{\hat \otimes n}$.
\end{itemize}
Clearly we have $T((\mathcal{H}))$ is the dual of $T(\mathcal{H})$ and $\textbf{T}((\mathcal{H}))$ is the dual of $\textbf{T}(\mathcal{H})$. We are interested in the following subset of $\textbf{T}((\mathcal{H}))$: for $\textbf{h} = \sum_{n=0}^\infty \textbf{h}_n$ with $\textbf{h}_n \in \mathcal{H}^{\hat \otimes n}$, let 
\begin{equation}\label{eq: Hilbert space in tensor algebra}
	\tilde{\mathcal{H}} := \Big\{\textbf{h} \in  \textbf{T}((\mathcal{H})): \sum_{n=0}^\infty \langle\textbf{h}_n, \textbf{h}_n\rangle_{\mathcal{H}^{\hat \otimes n}} < \infty \Big\},
\end{equation}
and $\langle \cdot, \cdot \rangle_{\tilde{\mathcal{H}}}$ be the inner product on $\tilde{\mathcal{H}}$ such that $\langle\textbf{h}, \textbf{h}^\prime \rangle_{\tilde{\mathcal{H}}} = \sum_{n=0}^\infty \langle\textbf{h}_n, \textbf{h}^\prime_n\rangle_{\mathcal{H}^{\hat \otimes n}}$. It is easy to see that $(\tilde{\mathcal{H}}, \langle\cdot, \cdot \rangle_{\tilde{\mathcal{H}}})$ is a separable Hilbert space.\\
In the present paper we will mainly consider the following special Hilbert spaces constructed as in \eqref{eq: Hilbert space in tensor algebra}:
\begin{itemize}
	\item $\mathcal{H}^0 := \mathbb{R}^d$,
	\item $\mathcal{H}^1 := \widetilde{\mathbb{R} \oplus \mathcal{H}^0} \subset \textbf{T}((\mathbb{R}^{d+1}))$,
	\item $\mathcal{H}^2 := \widetilde{\mathbb{R} \oplus \mathcal{H}^1} \subset \textbf{T}((\mathbb{R} \oplus \mathcal{H}^1))$,
	\item $\mathcal{H}^r := \widetilde{\mathbb{R} \oplus \mathcal{H}^{r-1}} \subset \textbf{T}((\mathbb{R} \oplus \mathcal{H}^{r-1}))$ for $r \ge 1$,
\end{itemize}
where $\mathbb{R} \oplus \mathcal{H}^{r}$ is equipped with the natural Hilbert space structure.\\
For a given Hilbert space $\mathcal{H}$, we use $\exp_{\otimes, \mathcal{H}}$ to denote the tensor exponential map from $\mathcal{H}$ to $T((\mathcal{H}))$. i.e.,
$$
\exp_{\otimes, \mathcal{H}}(h) := \sum_{n=0}^\infty \frac{h^{\otimes n}}{n!}.
$$
Clearly, thanks to the factorial decay, we have $\exp_{\otimes, \mathcal{H}}(h) \in \tilde{\mathcal{H}}$ for all $h \in \mathcal{H}$. Using the above notations, we set $\exp_{r} = \exp_{\otimes, \mathbb{R}\oplus \mathcal{H}^r}$ for every $r \ge 0$ such that for any $t \in \mathbb{R}$ and $h \in \mathcal{H}^r$, 
$$
\exp_{r}(t,h) = \sum_{n=0}^\infty \frac{(t,h)^{\otimes n}}{n!} \in \mathcal{H}^{r+1}.
$$
\begin{remark}
	For any given function $\phi: \mathbb{N} \cup \{0\} \to (0,\infty)$ one can actually define a $\phi$--inner product on the tensor algebra $\textbf{T}(\mathcal{H})$, namely 
	$$
	\langle \mathbf{h}, \mathbf{h}^\prime \rangle_{\phi} := \sum_{n=0}^\infty \phi(n) \langle \mathbf{h}_n, \mathbf{h}^\prime_n \rangle_{\mathcal{H}^{\hat \otimes n}}.
	$$
	In the present paper we only consider the classical case that $\phi \equiv 1$ (as in \cite{chevyrev2018signature}). For more examples of the weights $\phi$ and their applications we refer readers to \cite{Cass2021general}.
\end{remark}

\subsection{Algebraic properties of signatures}\label{subsect: algebra of higher rank sig}
First let us recall the definition of (time--augmented) signature for discrete--time path.

\begin{definition}\label{def: signature}
	Let $I = \{0 = t_0 < t_1 < \ldots < t_N = T\}$ and $x: I \to \mathcal{H}$ be a discrete--time path taking values in a Hilbert space $\mathcal{H}$. Then the time--augmented signature of $x$ is defined as
	$$
	S_{\mathcal{H}}(x) := \prod_{i=0}^N \exp_{\otimes, \mathbb{R} \oplus \mathcal{H}} (\Delta t_i, x_{t_{i-1},t_i}) \in \widetilde{\mathbb{R} \oplus \mathcal{H}}
	$$
	where $t_{-1} := 0$, $x_{t_{-1}} := 0$, $\Delta t_i = t_{i} - t_{i-1}$, $x_{t_{i-1},t_i} = x_{t_i} - x_{t_{i-1}}$ and $\widetilde{\mathbb{R} \oplus \mathcal{H}} \subset \textbf{T}((\mathbb{R} \oplus \mathcal{H}))$ is the Hilbert space defined as in Appendix \ref{subsect: hilbert space on tensor algebra}.
\end{definition}
\begin{remark}\label{remark: continuous time signature}
	The notion of (time--augmented) signature can be extended to continuous $\mathcal{H}$--valued paths defined on the continuous time interval $[0,T]$ with bounded variation. More precisely, for $x \in C^{1\text{--var}}([0,T],\mathcal{H})$, its signature is the collection of all iterated integrals
	\begin{equation}\label{eq: signature for path on continuous time}
		S_{\mathcal{H}}(x) = \sum_{n=0}^\infty S^n_{\mathcal{H}}(x), \quad   S^n_{\mathcal{H}}(x) := \int_{0<u_1<\ldots<u_n<T} d(u_1,x_{u_1}) \otimes \ldots \otimes d(u_n,x_{u_n}),
	\end{equation}
	where $S^0_{\mathcal{H}} \equiv 1$ and the integrals are given in the sense of Riemann--Stieltjes, see \cite[Sect. 2.2.2]{Lyons2007rough}. One can easily check that  the signature (see Definition \ref{def: signature}) of discrete--time path $x$ on $I$ actually coincides with the one given in terms of equation \eqref{eq: signature for path on continuous time} by extending $x$ to piecewise linear path $\tilde{x}$ on $[0,T]$ via classical linear interpolation.
\end{remark}
We also recall that there is a bilinear map defined on the graded vector space $T(\bar{\mathcal{H}})$ into itself, $\bar{\mathcal{H}} := \mathbb{R} \oplus \mathcal{H}$, which is denoted by $\shuffle$ and called the shuffle product. A key property of signature mapping (of discrete--time paths) is that as a finite product of tensor exponential, $S_{\mathcal{H}}(x)$ takes value in the character group inside the dual space $T((\bar{\mathcal{H}}))$, and the group property is translated into the following algebraic relation:
\begin{equation}\label{eq: shuffle product relation}
	\langle \textbf{h}_1 \shuffle \textbf{h}_2, S_{\mathcal{H}}(x) \rangle   =  \langle \textbf{h}_1,  S_{\mathcal{H}}(x)\rangle\langle  \textbf{h}_2,S_{\mathcal{H}}(x)  \rangle
\end{equation}
for all $\textbf{h}_1, \textbf{h}_2 \in T(\bar{\mathcal{H}})$, and $\langle \cdot,\cdot \rangle$ denotes the dual pairing of $T(\bar{\mathcal{H}})$ and $T((\bar{\mathcal{H}}))$. All above results can be found in \cite[Chapters 2,3]{Reutenauer1993Lie} as $T(\bar{\mathcal{H}})$ corresponds to the free Lie algebra over the alphabet set consisting of any Hamel basis of the vector space $\bar{\mathcal{H}}$. Of course signatures of continuous bounded variation paths also satisfy the relation \eqref{eq: shuffle product relation}, see e.g. \cite{Lyons2007rough}.\\
By \cite[Sect. 5.3]{chevyrev2018signature} we know that the time--augmented signature $S_{\mathcal{H}}$ is injective on the set of discrete--time $\mathcal{H}$--valued paths. Combining this with the relation \eqref{eq: shuffle product relation} we can conclude that
\begin{equation}\label{eq: point separating algebra}
	\{\langle \textbf{h}, S_{\mathcal{H}}(\cdot) \rangle:     \textbf{h} \in T(\bar{\mathcal{H}}) \} \subset \mathbb{R} ^{(\mathcal{H})^I} \text{ is a point--separating algebra}.
\end{equation}

\subsubsection{Analytic properties of signatures}\label{subsect: analytic signature}  
For $I = \{0=t_0 <t_1 < \ldots <t_N = T\}$, we now equip the set of discrete--time $\mathcal{H}$--valued paths (indexed by $I$) with the total variation norm $\|\cdot\|_{1\text{--var}}$ with $\|x\|_{1\text{--var}}:= \sum_{i=0}^N\|x_{t_i} - x_{t_{i-1}}\|_{\mathcal{H}}$. A well--known result in rough path theory, see \cite[Theorem 3.1.3]{Lyons2002control}, tells us that the (time--augmented) signature map $S_{\mathcal{H}}$ is locally Lipschitz continuous with respect to $\|\cdot\|_{1\text{--var}}$, namely for every $M>0$, there exists a constant $C = C(M,T)$ such that for all $x,y \in (\mathcal{H})^I$ with $\|x\|_{1\text{--var}}, \|y\|_{1\text{--var}} \le M$, 
\begin{equation}\label{eq: continuity of signature}
	\|S_{\mathcal{H}}(x) - S_{\mathcal{H}}(y)\|_{\widetilde{\mathbb{R}\oplus \mathcal{H}}} \le C(M,T)\|x - y\|_{1\text{--var}}.
\end{equation}

\section{Basics on Universal Feature Maps and Vector--valued RKHS}\label{sect: kernel learning}
Let $\mathcal{H}$ denote a Hilbert space and $\mathcal{Z}$ denote a topological space. 
\begin{definition}\label{def: universal features}(\cite[Definition 2.2]{chevyrev2018signature})
	A mapping $\Phi: \mathcal{Z} \mapsto \mathcal{H}$ is called a universal feature map if it satisfies one of the following equivalent conditions:
	\begin{enumerate}
		\item The set $\{\langle h, \Phi(\cdot) \rangle_{\mathcal{H}}: h \in \mathcal{H} \}$ is dense in $C_b(\mathcal{Z})$ with respect to the strict topology.
		\item The RKHS induced by the kernel $k_\Phi(\cdot,\cdot)$ is dense in $C_b(\mathcal{Z})$ with respect to the strict topology, where $k_\Phi(z,z^\prime) := \langle \Phi(z), \Phi(z^\prime) \rangle_{\mathcal{H}}$ for $z. z^\prime \in \mathcal{Z}$.
		\item For finite Borel measures $\mu$ and $\nu$ on $\mathcal{Z}$, $\mu = \nu$ iff $\int \Phi(z) \mu(d z) = \int \Phi(z) \nu(d z)$.
	\end{enumerate}
\end{definition}
In the above definition, the ``universality'' of $\Phi$ is taken with respect to the strict topology on the space $C_b(\mathcal{Z})$ of continuous bounded real--valued functions. This notion can be naturally extended to vector--valued (continuous bounded) functions equipped with the $L^2$--topology.
\begin{definition}\label{def: vector valued RKHS}(\cite[Definition 2.10]{park2020measure})
	An $\mathcal{H}$--valued RKHS on $\mathcal{Z}$ is a Hilbert space $\mathcal{G}_{\mathcal{H}}$ such that
	\begin{enumerate}
		\item every $g \in \mathcal{G}_{\mathcal{H}}$ is a function from $\mathcal{Z}$ to $\mathcal{H}$,
		\item for every $z \in \mathcal{Z}$ the evaluation map $\text{ev}_z(g) = g(z)$ is a bounded linear map from $\mathcal{G}_{\mathcal{H}}$ to $\mathcal{H}$.
	\end{enumerate}
\end{definition}
As for the scalar case, every $\mathcal{H}$--valued RKHS on $\mathcal{Z}$ admits a (positive--definite) $\mathcal{H}$--kernel $\Gamma: \mathcal{Z} \times \mathcal{Z} \to L(\mathcal{H})$ in the sense of \cite[Definition 2.11]{park2020measure} and conversely every $\mathcal{H}$--kernel $\Gamma$ determines a $\mathcal{H}$--valued RKHS. 

\begin{definition}\label{def: universal kernel}\cite[Definition 4.3]{park2020measure}
	A $\mathcal{H}$--kernel $\Gamma$ with $\mathcal{H}$--valued RKHS $\mathcal{G}_{\mathcal{H}}$ is universal if $\mathcal{G}_{\mathcal{H}}$ is a subset of $C_0(\mathcal{Z},\mathcal{H})$ (the space of continuous functions vanishing at infinity) and it is dense in $L^2(\mathcal{Z},\nu;\mathcal{H})$ for any probability measure $\nu$.
\end{definition}
Suppose there exits a scalar kernel $k_{\mathcal{Z}}$ on $\mathcal{Z}$, then 
\begin{equation}\label{eq: a nice vector valued kernel}
	\Gamma := k_{\mathcal{Z}}(\cdot,\cdot)\text{Id}_{\mathcal{H}}
\end{equation}
is an $\mathcal{H}$--kernel. In this case $\Gamma$ is universal in the sense of Definition \ref{def: universal kernel} if $k_{\mathcal{Z}}(\cdot,\cdot)$ is a universal scalar kernel in the sense of Definition \ref{def: universal features}, see \cite[Example 14]{Carmeli2010Mercer}. \\
Note that thanks to \eqref{eq: universality of signature} the signature kernel $k_S$ is a universal scalar kernel on pathspace, and consequently the kernel $k_{S}(\cdot,\cdot)\text{Id}_{\mathcal{H}}$ is universal in the sense of Definition \ref{def: universal kernel}. This fact was used in Section \ref{sect: estimation of CKME}.
\section{Optimal Stopping Problem for Processes with General Filtration}\label{sect: higher rank adapted topologies}
In this section we introduce higher rank signatures and show their relation with higher rank adapted topologies as well as their applications in solving OSP for filtered processes with general filtration. Throughout this sequel, all filtered processes are defined on a given discrete time interval $I = \{0 = t_0 < t_1 < \ldots < t_N = T\}$ and take values in a compact subset $\mathcal{K} \subset \mathbb{R}^d$.\\
Recall that (cf. Defintion \ref{def: higher rank prediction processes}) for $\mathbb{X} \in \mathcal{F}\mathcal{P}_I(\mathcal{K})$, for any $r \ge 1$ we can define its rank $r$ prediction process $\hat X^r$ recursively by setting
$$
\hat X^r_t := \mathbb{P}(\hat X^{r-1} \in \cdot| \mathcal{F}_t), 
\quad t \in I.
$$
Then the convergence in the rank $r$ adapted topology is given in terms of the weak convergence of the law of rank $r$ prediction processes. Thanks to such a hierarchical structure of prediction processes, we can apply a recursive way to construct rank $r$ signatures for rank $r$ prediction processes accordingly: Let $S^1 = S$ be the signature map given in Definition \ref{def: signature rank 0} which serves as the rank $1$ signature of $\hat X^0 := X$, then suppose that we have obtained the rank $r$ signature map $S^{r}$ for the rank $r-1$ prediction process $\hat X^{r-1}$ for some $r \ge 1$, then the rank $r+1$ signature map for the rank $r$ prediction process $\hat X^r$ will be defined as:
$$
S^{r+1}(\hat X^r) = S_{\mathcal{H}^r} ( t \mapsto \mathbb{E}[S^r(\hat X^{r-1})| \mathcal{F}_t] )
$$
where $\mathcal{H}^r$ denotes the codomain of $S^{r}$ and $S_{\mathcal{H}^r}$ is the signature map for $\mathcal{H}^r$-- valued path, see Appendix \ref{appendix: signature}. Note that since the signature map $S_{\mathcal{H}^k}$ is continuous for every $k \ge 1$ and $X = \hat X^0$ takes values in compact set $\mathcal{K}$, every conditional expectation $\mathbb{E}[S^r(\hat X^{r-1})| \mathcal{F}_t]$ is a well--defined $\mathcal{H}^r$-- valued Bochner integral.\\
Having the notion of rank $r$ signatures at hand, we can generalize Theorems \ref{thm: main theorem 1} and  \ref{thm: main theorem 2} to rank $r$ adapted topology for any $r \ge 0$.
\begin{theorem}(\cite[Theorem 4]{bonnier2020adapted})\label{thm: general MMD}
	For any $r \ge 0$ we have
	$$
	\mathbb{X} \sim _{r}\mathbb{Y} \iff \mathcal{L}(\hat X^{r}) = \mathcal{L}(\hat Y^{r}) \iff \mathbb{E}_{\mathbb{P}}[S^{r+1}(\hat X^{r})] = \mathbb{E}_{\mathbb{Q}}[S^{r+1}(\hat Y^{r})].
	$$
	Moreover, the rank $r+1$ Signature MMD 
	$$
	\mathcal{D}^{r+1}_{\mathcal{S}}(\mathbb{X}, \mathbb{Y}) := \| \mathbb{E}_{\mathbb{P}}[S^{r+1}(\hat X^r)] - \mathbb{E}_{\mathbb{Q}}[S^{r+1}(\hat Y^r)] \|_{\mathcal{H}^{r+1}}
	$$
	metrizes the rank $r$ adapted topology $\hat \tau_r$.
\end{theorem}
As we have mentioned in Remark \ref{remark: counterexample of OSP for general filtrations}, in the contrast to the case of filtered processes with natural filtration, the value functions in OSP for general filtered processes are only continuous for higher rank adapted topologies. 
\begin{theorem}\label{thm: equivalence of adapted wasserstein metric and higher rank MMD}
	Let $I = \{0 = t_0 < t_1 < \ldots < t_N =T\}$ and $\mathcal{K} \subset \mathbb{R}^d$ be a compact subset. Recall that $\Lambda_{\text{plain}}(\mathcal{K}) \subset \mathcal{F}\mathcal{P}_I(\mathcal{K})$ denotes the subset of all filtered processes with natural filtration. 
	\begin{enumerate}
		\item Restricted to $\Lambda_{\text{plain}}(\mathcal{K})$, the adapted Wasserstein metric $\mathcal{A}\mathcal{W}_p$ and the rank $2$ Signature MMD $\mathcal{D}^2_{\mathcal{S}}$ generated the same topology $\hat \tau_1$, for any $p \in [1,\infty)$.
		\item On the whole space $\mathcal{F}\mathcal{P}_I(\mathcal{K})$ the adapted Wasserstein metric $\mathcal{A}\mathcal{W}_p$ and the rank $N+1$ Signature MMD $\mathcal{D}^{N+1}_{\mathcal{S}}$ generate the same topology $\hat \tau_N$, for any $p \in [1,\infty)$.
		\item The value function $v$ of Optimal Stopping Problem  defined on the whole space $\mathcal{F}\mathcal{P}_I(\mathcal{K})$ is continuous for the rank $N+1$ Signature MMD $\mathcal{D}^{N+1}_{\mathcal{S}}$.
	\end{enumerate}
\end{theorem}
\begin{proof}
	The first statement follows from \cite[Theorem 1.2]{backhoff2020all}, which proved that the topology induced by adapted Wasserstein metrics $\mathcal{A}\mathcal{W}_p$ is  same as Aldous' extended weak topology $\hat \tau_1$ when filtrations are natural. The second claim follows from \cite[Lemmas 4.7, 4.10]{Bartl2021Wasserstein}, which implicitly showed that the topology for general filtered processes generated by $\mathcal{A}\mathcal{W}_p$ is same as the rank $N$ adapted topology $\hat \tau_N$ provided the discrete interval $I$ contains $N+1$ points and the state space $\mathcal{K}$ is compact, together with \cite[Theorem 4]{bonnier2020adapted} (which showed that the rank $N+1$ Signature MMD $\mathcal{D}^{N+1}_{\mathcal{S}}$ induces the topology $\hat \tau_N$). The last assertion follows from \cite[Example 4.5]{Bartl2021Wasserstein} (which proved that $v$ is continuous for the rank $N$ adapted topology $\hat \tau_N$) and the second claim above (the rank $N+1$ Signature MMD $\mathcal{D}^{N+1}_{\mathcal{S}}$ generate $\hat \tau_N$ when we have a compact state space).
\end{proof}
Thanks to the above results, our kernel learning approach (Theorem \ref{thm: main theorem 3}) remains valid for solving OSP for filtered processes with general filtration: one only need to replace the rank $2$ Signature MMD $\mathcal{D}_{\mathcal{S}}^2$ by the rank $N+1$ Signautre MMD $\mathcal{D}_{\mathcal{S}}^{r+1}$, provided the time interval $I$ contains $N+1$ points.
\begin{corollary}\label{cor: kernel learning of OSP for general case}
	Let $I = \{0 = t_0 < t_1 < \ldots < t_N = T\}$ and $\sigma > 0$ be fixed. We define a kernel $K^{N+1}_{\mathcal{S}}: \mathcal{F}\mathcal{P}_I(\mathcal{K}) \times \mathcal{F}\mathcal{P}_I(\mathcal{K}) \to \mathbb{R}$ via
	$$
	K^{N+1}_{\mathcal{S}}(\mathbb{X}, \mathbb{Y}) := \exp(-\sigma^2 \mathcal{D}_{\mathcal{S}}^{N+1}(\mathbb{X},\mathbb{Y})^2).
	$$
	Then the RKHS generated by $K^{N+1}_{\mathcal{S}}$ is dense in the space $C_b(\mathcal{F}\mathcal{P}_I(\mathcal{K}), \hat \tau_N)$ with respect to the uniform topology\footnote{In the contrast to the noncompactness of $\Lambda_{\text{plain}}(\mathcal{K})$, the set $\mathcal{F}\mathcal{P}_I(\mathcal{K})$ is compact for $\hat \tau_N$ for any $N$, see \cite[Theorem 1.7]{Bartl2021Wasserstein}}. In particular, if $v(\cdot)$ denotes the cost function in OSP, then for any $\varepsilon > 0$ and $\mathbb{X} \in \Lambda_{\text{plain}}(\mathcal{K})$, there exist an $M \in \mathbb{N}$, $\mathbb{Y}^i \in \Lambda_{\text{plain}}(\mathcal{K})$ and scalars $a_i \in \mathbb{R}$, $i=1,\ldots,M$ such that 
	$$
	\Big|v(\mathbb{X}) - \sum_{i=1}^M a_i \exp(-\sigma^2 \mathcal{D}^{N+1}_{\mathcal{S}}(\mathbb{X}, \mathbb{Y}^i)^2)\Big| \le \varepsilon.
	$$
\end{corollary}
Now, invoking Remark \ref{remark: generalization of CKME to higher order}, the above observation allows us to generalize our algorithm in Section \ref{sect: Regression algorithm} to all filtered processes with general filtration, as long as their filtrations are generated by some observable processes. More precisely, suppose $\mathbb{X} = (\Omega, \mathcal{F}, \mathbb{F}, \mathbb{P}, X)$ is an element in $\mathcal{F}\mathcal{P}_I(\mathcal{K})$ such that the filtration $\mathbb{F}$ is generated by some stochastic process $\xi = (\xi_t)_{t \in I}$, then we can construct an empirical estimator $\widehat{F_{t,K_S}}((\bar \xi^i)_{i \le M}, \cdot)$ for $\mathbb{E}_{\mathbb{P}}[k_S(X, \cdot)| \sigma(\xi_s, s\le t)]$ by sampling $M$ realizations $(\bar \xi^i)_{i \le M}$ from the law of $\xi$. Using this estimator we can construct the corresponding empirical distance $\widehat{\mathcal{D}_{\mathcal{S}}^{N+1}}$ for $\mathcal{D}^{N+1}_{\mathcal{S}}$ which in turn allows us to perform Distribution Regression for value functions of OSP defined on $\mathcal{F}\mathcal{P}_I(\mathcal{K})$ as before.
\section{Space--Time Discretization}\label{subsect: discretize}
In this Appendix we briefly discuss the convergence of the value functions in OSP from discrete time markets to continuous time market. The result obtained here may confirm, at least theoretically, that one can approximate the value functions in OSP for continuous time models via space--time discretization.\\ 
	Fix a time interval $[0,T]$, we consider the value function of Optimal Stopping Problem for  filtered processes defined on $[0,T]$ with continuous trajectories related to their natural filtration:
	$$
	v(\mathbb{X}) = \sup_{\tau \in \mathcal{T}_{\mathbb{X}}} \mathbb{E}_{\mathbb{P}}[\gamma(X_\tau)].
	$$
	Here we do not put any other additional assumptions on $\mathbb{X}$: $\mathbb{X}$ can be continuous semimartingales, fractional Brownian motions, and whatsoever. \\
	Now we discretize the market, both in space and in time, in the following simple way: for a partition $\mathcal{P}_n = \{0= t^n_0 < t^n_1 < \ldots < t^n_{N_n} = T\}$, define
	$$
	X^n := \sum_{k=0}^{N_n-1} (X^n_{t^k_n} \wedge n) \vee (-n)1_{[t^n_k,t^n_{k+1})} + (X^n_T \wedge n) \vee (-n)1_{\{T\}},
	$$
	and endow it with its natural filtration $\mathcal{F}^n_t = \sigma(X^n_s:s \le t)$. As a  simple consequence of \cite[Theorem 5]{Coquet2007converge} we can show that $v(\mathbb{X}^n) \to v(\mathbb{X})$ for $\mathbb{X}^n = (\Omega, \mathcal{F}^n_T, \mathbb{F}^n = (\mathcal{F}^n_t), X^n,\mathbb{P})$ as long as the mesh size of $\mathcal{P}_n$ tends to $0$.
	\begin{lemma}\label{lemma: convergence of discrete market to continuous time market}
		For any continuous time filtered process $\mathbb{X}$ with continuous trajectories and natural filtration , it holds that $v(\mathbb{X}^n) \to v(\mathbb{X})$ along any sequence of partitions $(\mathcal{P}_n)_{n \ge 1}$ with $|\mathcal{P}_n| \to 0$.
	\end{lemma}
	\begin{proof}
		Let $\mathbb{X}$ be a given continuous filtered process with the natural filtration. We fix a sequence of partitions of $[0,T]$ such that their mesh size $|\mathcal{P}_n| \to 0$ as $n \to \infty$. We denote $\mathcal{P}_n := \{0=t^n_0 < t^n_1<\ldots <t^n_{N_n} = T\}$.\\
		Now for every $n$ we define the localized discrete market $\mathbb{X}^n = (\Omega, \mathcal{F}^n_T, \mathbb{F}^n = (\mathcal{F}^n_t)_{t \in [0,T]}, X^n, \mathbb{P})$ with
		$$
		X^n = \sum_{k=0}^{N_n - 1}(X^n_{t^n_k} \wedge n)\vee (-n)1_{[t^n_k,t^n_{k+1})} + (X^n_{T} \wedge n)\vee (-n)1_{\{T\}}
		$$
		and $\mathcal{F}^n_t = \sigma(X^n_s: s\le t)$, $t \in [0, T]$ the natural filtration of $X^n$. Furthermore, let $\mathcal{T}$ be the set of all $\mathbb{F}$--stopping times.\\
		Since $X$ is continuous, we indeed have $\sup_{t \in [0,T]}|X^n_t - X_t| \to 0$ almost surely as $n \to \infty$, which further implies that 
		$$
		\lim_{\delta \to 0}\limsup_{n \to \infty} \sup_{\tau,\sigma \in \mathcal{T}, \sigma \le \tau \le \sigma + \delta} |X^n_\sigma - X^n_\tau| = 0
		$$
		almost surely, and therefore the Aldous' condition \cite[(1)]{Coquet2007converge} is satisfied for $(\mathbb{X}^n)_{n \in \mathbb{N}}$. Hence, by \cite[Theorem 5]{Coquet2007converge} we have $v(\mathbb{X}^n) \to v(\mathbb{X})$ as $n \to \infty$. In other words, for any $\varepsilon >0$ there exists an $N = N(\varepsilon)$ such that for all $n \ge N$ it holds that $$
		|v(\mathbb{X}^n) - v(\mathbb{X})| \le \varepsilon.
		$$
	\end{proof}
	Certainly the above result is purely theoretical and not quantitative. In practice, the choice of time--discretization may influence the computations cost and convergence speed crucially. It will be interesting to give a quantitative estimates for the convergence rate of value functions in terms of the mesh sizes of time--partitions and we leave it for a future work.

\section{Proofs}\label{appendix: Proof}
\begin{proof}[Proof of Theorem \ref{thm: main theorem 1}]
	Let us start with the case $r=0$. Since $X$ and $Y$ are (discrete) paths taking values in $\mathcal{K}$, the signature map $S$ is continuous (see e.g. \eqref{eq: continuity of signature}, the $\mathcal{H}^1$--valued Bochner integrals $\mathbb{E}_{\mathbb{P}}[S(X)]$ and $\mathbb{E}_{\mathbb{Q}}[S(Y)]$ are well--defined. Then using the algebraic property of signature $S$, by \eqref{eq: point separating algebra} we know that the class $\{\langle \ell, S(x)\rangle_{\mathcal{H}^1}: \ell \in T(\mathbb{R}^{d+1})\}$ is a point--separating algebra in $C(\mathcal{K}^I,\mathbb{R})$. So by a standard Stone--Weierstrass argument we obtain that 
	\begin{align*}
		\mathbb{P}_X = \mathbb{Q}_Y &\iff \mathbb{E}_{\mathbb{P}}[\langle \ell, S(X) \rangle_{\mathcal{H}^1}] = \mathbb{E}_{\mathbb{Q}}[\langle \ell, S(Y) \rangle_{\mathcal{H}^1}], \quad \forall \ell \in T(\mathbb{R}^{d+1})\\
		&\iff \mathbb{E}_{\mathbb{P}}[S(X)] = \mathbb{E}_{\mathbb{Q}}[S(Y)].
	\end{align*}
	For the case $r = 1$, we first apply the above result to every $\hat X^1_t = \mathbb{P}(X \in \cdot|\mathcal{F}_t)$, $t \in I$ to see that $\hat X^1 \mapsto (\mathbb{E}_{\mathbb{P}}[S(X)|\mathcal{F}_t])_{t \in I}$ is injective, that is, we can consider the $\mathcal{H}^1$--valued path $(\mathbb{E}_{\mathbb{P}}[S(X)|\mathcal{F}_t])_{t \in I}$ instead of the measure--valued path $\hat X^1$. Then, using \eqref{eq: point separating algebra} to signatures of paths evolving in the space $\mathcal{H}^1$ (i.e., $(\mathbb{E}_{\mathbb{P}}[S(X)|\mathcal{F}_t])_{t \in I}$) and the Stone--Weierstrass theorem again one can obtain that
	\begin{align*}
		\mathbb{X} \sim_1 \mathbb{Y} &\iff \mathcal{L}(\hat X^1) = \mathcal{L}(\hat Y^1) \\
		&\iff \mathbb{E}_{\mathbb{P}}[\langle \ell, S^2(\hat X^1) \rangle_{\mathcal{H}^2}] = \mathbb{E}_{\mathbb{Q}}[\langle \ell, S^2(\hat Y^1) \rangle_{\mathcal{H}^2}], \quad \forall \ell \in T(\mathbb{R}\oplus \mathcal{H}^1)\\
		&\iff \mathbb{E}_{\mathbb{P}}[S^2(\hat X^1)] = \mathbb{E}_{\mathbb{Q}}[S^2(\hat Y^1)],
	\end{align*}
	as $S^2(\hat X^1) = S_{\mathcal{H}^1}(t \mapsto \mathbb{E}_{\mathbb{P}}[S(X)|\mathcal{F}_t])$.
\end{proof}

\begin{proof}[Proof of Theorem \ref{thm: main theorem 2}]
	Since $\mathcal{K}$ is compact, the discrete path space $\mathcal{K}^I$ is compact, and therefore by \eqref{eq: point separating algebra} and the Stone--Weierstrass theorem every continuous function $f: \mathcal{K}^I \to \mathbb{R}$ can be uniformly approximated by $\langle \ell, S(x)\rangle_{\mathcal{H}^1}$ for $\ell \in T(\mathbb{R}^{d+1})$, and therefore
	\begin{align*}
		\mathbb{P}_{X^n} \to \mathbb{P}_X \text{ weakly} &\iff \mathbb{E}_{\mathbb{P}^n}[f(X^n)] \to \mathbb{E}_{\mathbb{P}}[f(X)] \quad \forall f \in C(\mathcal{K}^I, \mathbb{R}) \\
		&\iff \mathbb{E}_{\mathbb{P}^n}[\langle \ell, S(X^n) \rangle_{\mathcal{H}^1}] \to \mathbb{E}_{\mathbb{P}}[\langle \ell, S(X) \rangle_{\mathcal{H}^1}], \quad \forall \ell \in T(\mathbb{R}^{d+1})\\
		&\iff \mathcal{D}^1_{\mathcal{S}}(\mathbb{X}^n,\mathbb{X}) =  \|\mathbb{E}_{\mathbb{P}^n}[S(X^n)] \to \mathbb{E}_{\mathbb{P}}[S(X)]\|_{\mathcal{H}^1} \to 0.
	\end{align*}
	Similarly, since $\mathcal{P}(\mathcal{K}^I)$ is weakly compact and so is the product  space $\mathcal{P}(\mathcal{K}^I)^I$, as well as $\hat X^1 \in \mathcal{P}(\mathcal{K}^I)^I$, using \eqref{eq: point separating algebra} to $S^2$ as in the proof of Theorem \ref{thm: main theorem 1} we can again use the Stone--Weierstrass theorem to obtain
	\begin{align*}
		\mathcal{L}(\hat X^{n,1}) \to \mathcal{L}(\hat X) &\iff \mathbb{E}_{\mathbb{P}^n}[f(\hat X^{1,n})] \to \mathbb{E}_{\mathbb{P}}[f(\hat X^1)], \quad \forall f \in C(\mathcal{P}(\mathcal{K}^I)^I, \mathbb{R}) \\
		&\iff \mathbb{E}_{\mathbb{P}^n}[\langle \ell, S^2(X^{1,n}) \rangle_{\mathcal{H}^2}] \to \mathbb{E}_{\mathbb{P}}[\langle \ell, S^2(\hat X^1) \rangle_{\mathcal{H}^2}], \forall \ell \in T(\mathbb{R}\oplus \mathcal{H}^1)\\
		&\iff \mathcal{D}^2_{\mathcal{S}}(\mathbb{X}^n,\mathbb{X}) =  \|\mathbb{E}_{\mathbb{P}^n}[S^2(\hat X^{1,n})] \to \mathbb{E}_{\mathbb{P}}[S^2(\hat X^1)]\|_{\mathcal{H}^2} \to 0.
	\end{align*}
\end{proof}

\begin{proof}[Proof of Theorem \ref{thm: main theorem 3}]
	\begin{enumerate}
		\item By \cite[Thm. 2.2]{christmann2010universal} if $\mathfrak{K}$ is a compact metric space and $H$ is a separable Hilbert space such that there exists a continuous and injective map $\rho : \mathfrak{K} \to H$, then for any real analytic function $f:\mathbb{R} \to \mathbb{R}$  with strictly positive Taylor expansion coefficients such that its Taylor series converges globally,  the kernel $k: \mathfrak{K} \times \mathfrak{K} \to \mathbb{R}$ given by 
		\begin{equation*}
			k(z,z') = f(\langle \rho(z), \rho(z^\prime) \rangle_H)
		\end{equation*}
		is universal in the sense that its RKHS is dense in the space of continuous functions from $\mathfrak{K}$ to $\mathbb{R}$. \\
		Now, since $\mathcal{K}$ is compact, by \cite[Theorem 5.1, Lemma 4.7]{Bartl2021Wasserstein} we know that for any $r \ge 0$ the topological space $(\mathcal{F}\mathcal{P}_I(\mathcal{K}), \hat \tau_{r})$ is compact. Then, since $\mathbb{X} \mapsto \mathbb{E}_{\mathbb{P}}[S^2(\hat X)]$ is injective (see Theorem \ref{thm: main theorem 1}), the above result applied to $\mathfrak{K} = (\mathcal{F}\mathcal{P}_I(\mathcal{K}), \hat \tau_{1})$, $H = \mathcal{H}^{2}$ and $\rho = S^{2}$ yields that the RKHS induced by $K_f(\mathbb{X}, \mathbb{Y}) = f(\langle \mathbb{E}_{\mathbb{P}}[S^{2}(\hat X)], 
		\mathbb{E}_{\mathbb{Q}}[S^{2}(\hat Y)] \rangle_{\mathcal{H}^{2}})$ is dense in $C_b(\mathcal{F}\mathcal{P}_I(\mathcal{K}), \hat \tau_{1})$ for any $f$ satisfying the above mentioned conditions.
		
		\item Now let us consider the value function $v(\cdot)$ of OSP restricted on $\Lambda_{\text{plain}}(\mathcal{K})$, i.e., we only consider stopping times for the natural filtration. Without loss of generality we take $\sigma =1$. First of all, by \cite[Theorem 1.1]{backhoff2020all} we know that $v|_{\Lambda_{\text{plain}}(\mathcal{K})}$ is continuous for the rank $1$ adapted topology $\hat \tau_1$ on $\Lambda_{\text{plain}}(\mathcal{K})$, that is, $v|_{\Lambda_{\text{plain}}(\mathcal{K})}$ is an element in $C_b(\Lambda_{\text{plain}}(\mathcal{K}),\hat \tau_1)$. Now we take $f(x) = \exp(x)$, which is an analytic function on $\mathbb{R}$ with strictly positive Taylor expansion coefficients. In view of the result obtained in the last step we know that the RKHS induced by $K_{\exp}(\mathbb{X},\mathbb{Y}) = \exp \circ \langle \mathbb{E}_{\mathbb{P}}[S^2(\hat X^1)], \mathbb{E}_{\mathbb{Q}}[S^2(\hat Y^1)] \rangle_{\mathcal{H}_{2}}$ is dense in $C_b((\Lambda_{\text{plain}}(\mathcal{K}),\hat \tau_1), \mathbb{R})$ with respect to the strict topology. Here we only need to note that we are using Giles' Theorem (see \cite[Theorem 2.6]{chevyrev2018signature}) on the non--compact set $(\Lambda_{\text{plain}}(\mathcal{K}),\hat \tau_1)$ instead of the classical Stone--Weierstrass theoreom on compact sets. \\
		Furthermore, let $\alpha(\mathbb{X}) := K_{\exp}(\mathbb{X},\mathbb{X})^{-\frac{1}{2}}$ for $\mathbb{X} \in \Lambda_{\text{plain}}(\mathcal{K})$. In view of the proof of \cite[Theorem 4.55]{christmann2008SVM} we see that if $\Phi$ is the canonical feature map for $K_{\exp}$, then $\alpha \Phi$ is the canonical feature map for the kernel $K^2_{\mathcal{S}}(\mathbb{X}, \mathbb{Y}) := \exp(-\mathcal{D}^2_{\mathcal{S}}(\mathbb{X},\mathbb{Y})^2) = \exp(-\|\mathbb{E}_{\mathbb{P}}[S^2(\hat X^1)] - \mathbb{E}_{\mathbb{Q}}[S^2(\hat Y^1)]\|_{\mathcal{H}^2}^2)$. Here we note that since $(\Lambda_{\text{plain}}(\mathcal{K}),\hat \tau_1) \subset (\mathcal{F}\mathcal{P}_I(\mathcal{K}), \hat \tau_{1})$ and the latter set is compact, the continuity of $\alpha$ on $(\mathcal{F}\mathcal{P}_I(\mathcal{K}), \hat \tau_{1})$ implies that $c:=\sup_{\mathbb{X} \in \Lambda_{\text{plain}}(\mathcal{K})} |\alpha(\mathbb{X})| < \infty$. Since the RKHS induced by $K_{\exp}(\mathbb{X},\mathbb{Y}) = \exp \circ \langle \mathbb{E}_{\mathbb{P}}[S^2(\hat X^1)], \mathbb{E}_{\mathbb{Q}}[S^2(\hat Y^1)] \rangle_{\mathcal{H}_{2}}$ is dense in $C_b(\Lambda_{\text{plain}}(\mathcal{K}),\hat \tau_1)$ with respect to the strict topology, for each $g \in C_b((\Lambda_{\text{plain}}(\mathcal{K}),\hat \tau_1)$, each function $\phi$ vanishing at infinity and each $\varepsilon > 0$ there exists an $\ell \in \mathcal{H}_2$  such that 
		$$
		\sup_{\mathbb{X} \in \Lambda_{\text{plain}}(\mathcal{K})}\Big|\Big(\langle \ell, \Phi(\mathbb{X})\rangle_{\mathcal{H}_2} - \frac{g}{\alpha}\Big)\phi(\mathbb{X})\Big| < \varepsilon/c,
		$$
		due to the definition of the strict topology. This in turn implies that
		\begin{align*}
			\sup_{\mathbb{X} \in \Lambda_{\text{plain}}(\mathcal{K})}\Big|\Big(\langle \ell, \alpha(\mathbb{X})\Phi(\mathbb{X})\rangle_{\mathcal{H}_2} - &g(\mathbb{X})\Big)\phi(\mathbb{X})\Big| \le \sup_{\mathbb{X} \in \Lambda_{\text{plain}}(\mathcal{K})} |\alpha(\mathbb{X})| \\
			&\times \sup_{\mathbb{X} \in \Lambda_{\text{plain}}(\mathcal{K})}\Big|\Big(\langle \ell, \Phi(\mathbb{X})\rangle_{\mathcal{H}_2} - \frac{g}{\alpha}\Big)\phi(\mathbb{X})\Big| \\
			&\quad \quad \quad \quad \quad \le c \frac{\varepsilon}{c} = \varepsilon.
		\end{align*}
		Hence, $\alpha \Phi$ is a universal feature map for the strict topology and our claim follows immediately from the fact that $\alpha \Phi$ induces the kernel $\exp(-\mathcal{D}^2_{\mathcal{S}}(\cdot,\cdot)^2)$ together with \cite[Proposition E.3]{chevyrev2018signature}.
		
	\end{enumerate}
\end{proof}

\begin{proof}[Proof of Lemma \ref{lemma: MMD by RKHS}]
	\textit{Step 1}: Thanks to the reproducing property that 
	$$
	\langle k_{S_{\mathcal{H}_{\mathcal{S}}}}(\hat x, \cdot), k_{S_{\mathcal{H}_{\mathcal{S}}}}(\hat y,\cdot) \rangle_{\mathcal{H}_{\mathcal{S}}} = \langle S_{\mathcal{H}_{\mathcal{S}}}(\hat x), S_{\mathcal{H}_{\mathcal{S}}}( \hat y) \rangle_{\widetilde{\mathbb{R} \oplus \mathcal{H}_{\mathcal{S}}}},
	$$
	by the same argument as we deduced the equation for $\mathcal{D}^1_{\mathcal{S}}$, we obtain that
	\begin{align*}
		\|\mu^2_{\hat X} - \mu^2_{\hat Y}\|_{\mathcal{H}_{\mathcal{S}}^2}^2  &= \mathbb{E}_{X,X^\prime}[\langle S_{\mathcal{H}_{\mathcal{S}}}(\mu_{\hat X^1}), S_{\mathcal{H}_{\mathcal{S}}}(\mu_{\hat X^{\prime,1}}) \rangle_{\widetilde{\mathbb{R} \oplus \mathcal{H}_{\mathcal{S}}}}] \\
		& \quad + \mathbb{E}_{Y,Y^\prime}[\langle S_{\mathcal{H}_{\mathcal{S}}}(\mu_{\hat Y^1}), S_{\mathcal{H}_{\mathcal{S}}}(\mu_{\hat Y^{\prime,1}}) \rangle_{\widetilde{\mathbb{R} \oplus \mathcal{H}_{\mathcal{S}}}}] \\
		&\quad -2\mathbb{E}_{X,Y}[\langle S_{\mathcal{H}_{\mathcal{S}}}(\mu_{\hat X^1}), S_{\mathcal{H}_{\mathcal{S}}}(\mu_{\hat Y^1}) \rangle_{\widetilde{\mathbb{R} \oplus \mathcal{H}_{\mathcal{S}}}}],
	\end{align*}
	where $\mathbb{E}_{X,X^\prime}$ denotes the expectation taken with respect to the product measure $\mathbb{P} \otimes \mathbb{P}$ defined on the filtered probability space $(\Omega \times \Omega, \mathbb{F} \otimes \mathbb{F})$ with $X$ and $X^\prime$ being the first and second coordinate process respectively, and the notations $\mathbb{E}_{Y,Y^\prime}$, $\mathbb{E}_{X,Y}$ and $\mu_{\hat X^{\prime,1}}$ are then self--explanatory.\\
	On the other hand, using exactly the same calculation, we also have 
	\begin{align*}
		\mathcal{D}_{\mathcal{S}}^2(\mathbb{X},  \mathbb{Y})^2 &= 
		\mathbb{E}_{X,X^\prime}[\langle S^2(\hat X^1), S^2(\hat X^{\prime,1})\rangle_{\mathcal{H}^2}] \\
		&\quad + \mathbb{E}_{Y,Y^\prime}[\langle S^2(\hat Y^1), S^2(\hat Y^{\prime,1}) \rangle_{\mathcal{H}^2}]\\
		& \quad -2\mathbb{E}_{X,Y}[\langle S^2(\hat X^1), S^2(\hat Y^1)\rangle_{\mathcal{H}^2}].
	\end{align*}
	Hence, it suffices to prove that
	$$\mathbb{E}_{X,X^\prime}[\langle S_{\mathcal{H}_{\mathcal{S}}}(\mu_{\hat X^1}), S_{\mathcal{H}_{\mathcal{S}}}(\mu_{\hat X^{\prime,1}}) \rangle_{\widetilde{\mathbb{R} \oplus \mathcal{H}_{\mathcal{S}}}}] = \mathbb{E}_{X,X^\prime}[\langle S^2(\hat X^1), S^2(\hat X^{\prime,1})\rangle_{\mathcal{H}^2}],
	$$
	$\mathbb{E}_{Y,Y^\prime}[\langle S_{\mathcal{H}_{\mathcal{S}}}(\mu_{\hat Y^1}), S_{\mathcal{H}_{\mathcal{S}}}(\mu_{\hat Y^{\prime,1}}) \rangle_{\widetilde{\mathbb{R} \oplus \mathcal{H}_{\mathcal{S}}}}] = \mathbb{E}_{Y,Y^\prime}[\langle S^2(\hat Y^1), S^2(\hat Y^{\prime,1}) \rangle_{\mathcal{H}^2}]$ and
	$$
	\mathbb{E}_{X,Y}[\langle S_{\mathcal{H}_{\mathcal{S}}}(\mu_{\hat X^1}), S_{\mathcal{H}_{\mathcal{S}}}(\mu_{\hat Y^1}) \rangle_{\widetilde{\mathbb{R} \oplus \mathcal{H}_{\mathcal{S}}}}] = \mathbb{E}_{X,Y}[\langle S^2(\hat X^1), S^2(\hat Y^1)\rangle_{\mathcal{H}^2}].
	$$
	\textit{Step 2}: Let us verify that $$\mathbb{E}_{X,Y}[\langle S_{\mathcal{H}_{\mathcal{S}}}(\mu_{\hat X^1}), S_{\mathcal{H}_{\mathcal{S}}}(\mu_{\hat Y^1}) \rangle_{\widetilde{\mathbb{R} \oplus \mathcal{H}_{\mathcal{S}}}}] = \mathbb{E}_{X,Y}[\langle S^2(\hat X^1), S^2(\hat Y^1)\rangle_{\mathcal{H}^2}].
	$$
	Then the other two equations can be shown analogously. 
	To do this, we first note that the signatures of paths on the discrete time interval $I$ (see Definition \ref{def: signature rank 0}) is equal to  the continuous path signatures (given by iterated integrals) of their \textit{linear interpolation}  on $[0,T]$, see Remark \ref{remark: continuous time signature}. From the latter  perspective and \cite[Theorem 2.5]{cass2020computing} we know that for any $\omega, \omega^\prime$, the inner product of two signatures $\langle S_{\mathcal{H}_{\mathcal{S}}}(\mu_{\hat X^1})(\omega), S_{\mathcal{H}_{\mathcal{S}}}(\mu_{\hat Y^1})(\omega^\prime) \rangle_{\widetilde{\mathbb{R} \oplus \mathcal{H}_{\mathcal{S}}}}$ is the unique solution $u(T,T)$ to the Goursat type PDE (with a given boundary condition)
	\begin{equation}\label{eq: first Goursat pde}
		\frac{\partial^2 u}{\partial t\partial s}(t,s) = \langle (1,\frac{\partial}{\partial t}\mu_{\hat X^1_t}(\omega)), (1,\frac{\partial}{\partial s}\mu_{\hat Y^1_s}(\omega^\prime)) \rangle_{\mathbb{R} \oplus \mathcal{H}_{\mathcal{S}}}u(t,s).
	\end{equation}
	On the other hand, recalling the construction of $S^2$ (see Definition \ref{def: rank 2 sig}), namely 
	$$
	S^2(\hat X^1) = S_{\mathcal{H}^1}( t \mapsto \mathbb{E}_{\mathbb{P}}[S(X)|\mathcal{F}_t]),
	$$
	using the same argument we also have $\langle S^2(\hat X^1)(\omega), S^2(\hat Y^1)(\omega^\prime)\rangle_{\mathcal{H}^2}$ is the unique solution $w(T,T)$ to the Goursat PDE (with the same boundary condition as above)
	\begin{equation}\label{eq; second Goursat pde}
		\frac{\partial^2 w}{\partial t\partial s}(t,s) = \langle (1,\frac{\partial}{\partial t}\mathbb{E}_{\mathbb{P}}[S(X)|\mathcal{F}_t](\omega)), (1,\frac{\partial}{\partial s}\mathbb{E}_{\mathbb{Q}}[S(Y)|\mathcal{G}_s](\omega^\prime)) \rangle_{\mathbb{R} \oplus \mathcal{H}^1}w(t,s).
	\end{equation}
	Since these paths are piecewise linear on fixed time interval, that is, for all $t \in (t_{i-1},t_i)$ it holds that
	$$
	\frac{\partial}{\partial t}\mu_{\hat X^1_t}(\omega) = \frac{\mu_{\hat X^1_{t_i}}(\omega) - \mu_{\hat X^1_{t_{i-1}}}(\omega)}{t_i - t_{i-1}}
	$$
	and similarly for $\frac{\partial}{\partial t}\mathbb{E}_{\mathbb{P}}[S(X)|\mathcal{F}_t](\omega)$. Based on these observations, we see that if we can prove that for all $s,t \in I$
	\begin{equation}\label{eq: two inner products are same}
		\langle \mu_{\hat X^1_t}(\omega),      \mu_{\hat Y^1_s}(\omega^\prime) \rangle_{\mathcal{H}_{\mathcal{S}}} = \langle \mathbb{E}_{\mathbb{P}}[S(X)|\mathcal{F}_t](\omega), \mathbb{E}_{\mathbb{Q}}[S(Y)|\mathcal{G}_s](\omega^\prime) \rangle_{\mathcal{H}^1},
	\end{equation}
	then $u$ in \eqref{eq: first Goursat pde} and $w$ in \eqref{eq; second Goursat pde} solve the same PDE, and so by the uniqueness of the solution to Goursat type PDE we can conclude that
	$$
	\langle S^2(\hat X^1)(\omega), S^2(\hat Y^1)(\omega^\prime)\rangle_{\mathcal{H}^2} = \langle S_{\mathcal{H}_{\mathcal{S}}}(\mu_{\hat X^1})(\omega), S_{\mathcal{H}_{\mathcal{S}}}(\mu_{\hat Y^1})(\omega^\prime) \rangle_{\widetilde{\mathbb{R} \oplus \mathcal{H}_{\mathcal{S}}}}
	$$
	which in turn implies that their expectations are same, and finishes our proof.\\
	\textit{Step 3}: Finally, let us check \eqref{eq: two inner products are same}. Since $$\mu_{\hat X^1_t}(\omega) = \mathbb{E}_{\mathbb{P}}[k_S(X, \cdot)|\mathcal{F}_t](\omega) = \int k_S(x, \cdot) \hat X^1_t(\omega)(dx)
	$$ 
	and $\mu_{\hat Y^1_s}(\omega^\prime) = \mathbb{E}_{\mathbb{Q}}[k_S(Y, \cdot)|\mathcal{G}_s](\omega^\prime) = \int k_S(y, \cdot) \hat Y^1_s(\omega^\prime)(dy)$,
	we indeed obtain that 
	$$
	\langle \mu_{\hat X^1_t}(\omega),      \mu_{\hat Y^1_s}(\omega^\prime) \rangle_{\mathcal{H}_{\mathcal{S}}} = \int \langle k_S(x, \cdot), k_S(y, \cdot)\rangle_{\mathcal{H}_{\mathcal{S}}} \hat X^1_t(\omega)(dx) \otimes \hat Y^1_s(\omega^\prime)(dy).
	$$
	However, in view of \eqref{eq: reproducing property} the above equation is equal to
	$$
	\int \langle S(x), S(y)\rangle_{\mathcal{H}^1} \hat X^1_t(\omega)(dx) \otimes \hat Y^1_s(\omega^\prime)(dy),
	$$
	which is nothing else but $\langle \mathbb{E}_{\mathbb{P}}[S(X)|\mathcal{F}_t](\omega), \mathbb{E}_{\mathbb{Q}}[S(Y)|\mathcal{G}_s](\omega^\prime) \rangle_{\mathcal{H}^1}$ (recall that $\hat X^1_t = \mathbb{P}(X \in \cdot|\mathcal{F}_t)$ and $\hat Y^1_s = \mathbb{Q}(Y \in \cdot|\mathcal{G}_s)$).
\end{proof}

\begin{proof}[Proof of Proposition \ref{prop: main result}]
	\begin{enumerate}
		\item  Using Theorem \ref{thm: main theorem 3} we deduce that there exist $\mathbb{Y}^{i} \in \Lambda_{\text{plain}}(\mathcal{K})$, $i \le M$, scalars $a_i, i \le M$ such that
		$$
		\Big|v(\mathbb{X}) - \sum_{i=1}^M a_i K^2_{\mathcal{S}}(\mathbb{X},\mathbb{Y}^{i})\Big|\le \varepsilon.
		$$
		where $K^2_{\mathcal{S}}(\cdot,\cdot) = \exp(-\sigma^2 \mathcal{D}^{2,n}_{\mathcal{S}}(\cdot,\cdot))$. In Section \ref{sect: estimation of CKME} we have constructed a consistent empirical estimator $\widehat{\mathcal{D}}^2_{\mathcal{S}}(\mathbb{X},\mathbb{Y})$ for the $2$--nd MMD $\mathcal{D}^2_{\mathcal{S}}(\mathbb{X},\mathbb{Y})$ out of empirical observations $x^i \sim \mathbb{P}_X$, $i\le m$ and $y^j \sim \mathbb{Q}_{Y}$, $j \le n$ so that $ \widehat{\mathcal{D}}^2_{\mathcal{S}}(\mathbb{X},\mathbb{Y}) \to \mathcal{D}^2_{\mathcal{S}}(\mathbb{X},\mathbb{Y})$ in probability as $m,n \to \infty$, see Theorem \ref{thm: convergence of empirical 2nd MMD}. As a consequence, we obtain a consistent empirical estimator $\widehat{K}^2_{\mathcal{S}}(\mathbb{X},\mathbb{Y}) = \exp(-\sigma^2\widehat{\mathcal{D}}^2_{\mathcal{S}}(\mathbb{X},\mathbb{Y}))$ for the kernel  $K^2_{\mathcal{S}}(\mathbb{X},\mathbb{Y})$ and $\hat{v}(\mathbb{X}) := \sum_{i=1}^M a_i \widehat{K}^2_{\mathcal{S}}(\mathbb{X},\mathbb{Y}^{i})$ can approximate the value $v(\mathbb{X})$ arbitrarily well. Note that the linear regression coefficients $a_i$ can be computed via Kernel Ridge Regression, see Section \ref{sect: summary}.
		\item In view of the algorithm developed in Section \ref{sect: Regression algorithm}, to compute $v(\mathbb{X})$ one only needs to solve a family of linear PDEs of the form 
		$$
		\frac{\partial^2 u_{x,y}}{\partial s \partial t}  = \big\langle \frac{\partial}{\partial t}x_t, \frac{\partial}{\partial s}y_{s}\big\rangle_{\mathbb{R}^d} u_{x,y}
		$$
		for paths $x,y$ taking values in $\mathbb{R}^d$. This completes the proof.
	\end{enumerate}
\end{proof}

\bibliography{references}{}
\bibliographystyle{amsalpha}

\end{document}